\documentclass[10pt, reqno]{amsart}

\usepackage[utf8]{inputenc}
\usepackage{amsmath}
\usepackage{amssymb}
\usepackage{mathrsfs}
\usepackage{graphics}
\usepackage[dvipdf]{graphicx}
\usepackage{geometry}
\usepackage{setspace}
\usepackage{color}
\usepackage{amsthm}
\usepackage{paralist}
\usepackage{dsfont}
\usepackage{verbatim}
\usepackage{bm}
\usepackage[numbers]{natbib}

\renewcommand{\tilde}{\widetilde}
\renewcommand{\hat}{\widehat}
\renewcommand{\bar}{\overline}

\theoremstyle{definition}
\newtheorem{definition}{Definition}[section]
\newtheorem{remark}[definition]{Remark}
\newtheorem{assumption}[definition]{Assumption}
\theoremstyle{plain}
\newtheorem{lemma}[definition]{Lemma}
\newtheorem{thm}[definition]{Theorem}

\numberwithin{equation}{section}

\newcommand\R{\mathbb{R}}

\newcommand\Ru{{\mathbb{R} \cup \left\{\infty\right\}}}

\newcommand\Nzero{\mathbb{N}_0}
\newcommand\Rd{{\mathbb{R}^d}}

\newcommand\ninN{{n\in\mathbb{N}}}
\newcommand\kinN{{k\in\mathbb{N}}}
\newcommand\ntoinf{{n\rightarrow \infty}}
\newcommand\ktoinf{{k\rightarrow \infty}}

\newcommand\dar{{\ \stackrel{\bd}{\rightarrow}} \ }
\newcommand\sar{{ \ \stackrel{\bs}{\rightharpoonup}\ }  }

\newcommand\rk{\rho^{k}_{\btau}}
\newcommand\rkm{\rho^{k-1}_{\btau}}

\newcommand\uk{u^k_{\btau}}
\newcommand\ukm{u^{k-1}_{\btau}}
\newcommand\ukp{u^{k+1}_{\btau}}
\newcommand\ukn{u^N_{\btau}}

\newcommand\ukw{u^k_{\btau,\omega}}
\newcommand\ukmw{u^{k-1}_{\btau,\omega}}
\newcommand\uknw{u^N_{\btau,\omega}}

\newcommand\tk{{t^k_\btau}}

\newcommand\tkp{{t^{k+1}_\btau}}

\newcommand\dk{\bd^2(\uk,\ukm)}
\newcommand\dks{\bd^2(u_*,\uk)}
\newcommand\dksm{\bd^2(u_*, \ukm)}
\newcommand\dksn{\bd^2(u_*,\ukn)}
\newcommand\dksz{\bd^2(u_*,u_0)}
\newcommand\ek{\mathcal{E}(\uk)}
\newcommand\ekm{\mathcal{E}(\ukm)}
\newcommand\ekn{\mathcal{E}(\ukn)}

\newcommand\eknw{\mathcal{E}(\uknw)}
\newcommand\pk{\mathcal{P}_{ t^k_\btau }(\uk)}
\newcommand\pkm{\mathcal{P}_{ t^k_\btau}(\ukm)}
\newcommand\pkw{\mathcal{P}_{ \omega t^k_\btau }(\ukw)}
\newcommand\pkmw{\mathcal{P}_{ \omega t^k_\btau}(\ukmw)}

\newcommand\dkw{\bd^2(\ukw,\ukmw)}

\newcommand\Pzwei{{{\mathscr{P}}_2 \left( \Rd \right)}}

\newcommand\Paz{{{\mathscr{P}}_{2,ac} \left( \Rd \right)}}

\newcommand\Lzwei{{{L^2}\left(\Rd\right)}}

\newcommand\Km{{\mathcal{K}_m}}

\newcommand{\btau}{{\bm{\tau}}}

\newcommand{\W}{\mathbf{W}_2}
\newcommand{\M}{\mathbf{M}_2}
\newcommand{\tW}{{\mathcal{W}}}

\newcommand{\N}{{\mathbb{N}}}

\newcommand{\bp}{{\bm{p}}}
\newcommand{\bd}{{\bm{d}}}
\newcommand{\bs}{{\bm{\sigma}}}
\newcommand{\X}{{\bm{X}}}

\newcommand{\dr}{{\ \mathrm{d} r }}
\newcommand{\dt}{{\ \mathrm{d} t }}
\newcommand{\dx}{{\ \mathrm{d} x }}
\newcommand{\dy}{{\ \mathrm{d} y }}
\newcommand{\dxy}{{\ \mathrm{d} x \mathrm{d} y }}
\newcommand{\ds}{{\ \mathrm{d} \sigma }}
\newcommand{\dss}{{\ \mathrm{d} s }}

\newcommand{\dm}{{\ \mathrm{d} \mu}}

\newcommand{\dd}{\,\mathrm{d}}

\newcommand{\ent}{\mathcal{E}}

\begin{document}


\begin{abstract}
We study the high-frequency limit of non-autonomous gradient flows in metric spaces of energy functionals comprising an explicitly time-dependent perturbation term which might oscillate in a rapid way, but fulfills a certain Lipschitz condition. On grounds of the existence results by Ferreira and Guevara (2015) on non-autonomous gradient flows (which we also extend to the framework of geodesically non-convex energies), we prove that the associated solution curves converge to a solution of the time-averaged evolution equation in the limit of infinite frequency. Under the additional assumption of dynamical geodesic $\lambda$-convexity of the energy, we obtain an explicit rate of convergence. In the non-convex case, we specifically investigate nonlinear drift-diffusion equations with time-dependent drift which are gradient flows with respect to the $L^2$-Wasserstein distance. We prove that a family of weak solutions obtained as a limit of the Minimizing Movement scheme exhibits the above-mentioned behaviour in the high-frequency limit.
\end{abstract}


\title[High-frequency limit]{High-frequency limit of non-autonomous gradient flows}
\author[Simon Plazotta]{Simon Plazotta}
\author[Jonathan Zinsl]{Jonathan Zinsl}
\address{Zentrum f\"ur Mathematik \\ Technische Universit\"at M\"unchen \\ 85747 Garching, Germany}
\email{plazotta@ma.tum.de}
\email{zinsl@ma.tum.de}
\thanks{This research has been supported by the German Research Foundation (DFG), SFB TRR 109. The authors would like to thank Daniel Matthes for helpful discussions and remarks.}
\date{\today}
\maketitle


\section{Introduction}\label{sec:intro}
In this work, we study \emph{non-autonomous evolution problems} of, for instance, the form of a \emph{nonlinear Fokker-Planck equation}:
\begin{align}
		\partial _t \rho (t,x)  &= \Delta_x \rho^m (t,x) + \mathrm{div}_x \left[\rho(t,x) \nabla_x (W_{\omega t}(x,\cdot) \ast \rho(t,\cdot)) \right], \label{FPeq} \\
\rho(0,x)&=\rho_0(x),\label{eq:icond}
\end{align}
where the sought-for solution $\rho:[0,\infty)\times\Rd\to [0,\infty]$ is nonnegative and preserves the initial mass. The key element of equation \eqref{FPeq} is the explicit time-dependence of the right-hand side via the potential $W_{(\cdot)}:[0,\infty)\times\Rd\times\Rd\to \R$ which comprises both confinement and interaction potentials. Above, $m\ge 1$ denotes the exponent of (possibly nonlinear) diffusion and $\omega>0$ is a parameter which reflects the frequency of oscillation of $W$, as $W$ is assumed to be $1$-periodic in its first argument. Our main interest lies in the behaviour of \eqref{FPeq} when $\omega\to\infty$: in a nutshell, we prove that the family of solutions $(\rho_\omega)_{\omega>0}$ converges to a solution $\rho$ of the limit problem
\begin{align}
		\partial _t \rho (t,x)  &= \Delta_x \rho^m (t,x) + \mathrm{div}_x \left[\rho(t,x) \nabla_x (\bar W(x,\cdot) \ast \rho(t,\cdot)) \right], \label{limFPeq}\\
\rho(0,x)&=\rho_0(x),\label{eq:limicond}
\end{align}
where the \emph{time-independent} limit potential $\bar W:\R^d\times\Rd\to\R$ is given as the time-average
\begin{align}\label{eq:limpot}
\bar W(x,y):=\int_0^1 W_t(x,y)\dd t,\qquad\text{at each }x,y\in\Rd.
\end{align}
We assume the following:
\begin{assumption}[Properties of $W$]\label{DefRegWintro}~
		\begin{enumerate}[({W}1)]
			\item  $W \in \mathcal{C}^{(1,2,2)} \left( \left[0,\infty\right) \times \Rd \times \Rd \right) $ is symmetric w.r.t. its second and third argument, and $1$-periodic in its first argument.
			\item  There exists non-negative constants $d_1, d_2 $ such that 
							\begin{align*}
			\left| W_t(x,y) \right| \leq  d_1 \left( 1 + \left\|x\right\|^2 + \left\|y\right\|^2 \right), \qquad \left| \partial_t W_t(x,y) \right| \leq d_2 \left( 1 + \left\|x\right\|^2 + \left\|y\right\|^2 \right).
			\end{align*}
			\item There exists a suitable non-negative function $\alpha \in  L^1_{\mathrm{loc}} \left( \left[0,\infty \right) \right)$ such that
			\begin{align*}
			\left| W_t(x,y) - W_s(x,y) \right| \leq \int_s^t \alpha(r) \dr   \left( 1 + \left\|x\right\|^2 + \left\|y\right\|^2 \right).
			\end{align*}
				\item For all $T>0$ and $\Omega \Subset \Rd$ (i.e., $\Omega$ is compactly contained in $\Rd$) there exists a non-negative constant $d_3$, an exponent $r<2$, and a non-negative function $\tilde{\alpha}\in L^1_\mathrm{loc}\left(\left[0,\infty \right) \right)$ such that for all $y\in\Rd, x\in \Omega, s \leq t\in \left[0,T\right]$: 
					\begin{align*}
					\left| \nabla_x W_t(x,y) - \nabla_x W_s(x,y) \right|  \leq \int_s^t \tilde{\alpha} (r) \dr   \left( 1 + \left\|y\right\|^2\right),		\qquad			\left| \nabla_x W_t(x,y) \right| &\leq  d_3 \left( 1 + \left\|y\right\|^{r} \right).
					\end{align*}	
					\item There exists a non-negative constant $d_4$ such that
				\begin{align*}
					\left| \Delta_x W_t(x,y) \right| \leq d_4(1+ \|x\|^2+\|y\|^2 ).
				\end{align*}
\item There exists a non-negative constant $L$ \emph{independent of $t$} such that for all $x,\tilde x,y,\tilde y \in \Rd$:
				\begin{align*}
					\left| ( W_t(x,y) - \overline{W}(x,y) ) - ( W_t(\tilde x,\tilde y) - \overline{W}(\tilde x,\tilde y) ) \right| \leq L \left( \left\|x-\tilde x\right\| +\left\|y-\tilde y\right\| \right).
				\end{align*}
		\end{enumerate}
\end{assumption}
We prove the existence of weak solutions to system \eqref{FPeq} using its formal gradient flow structure with respect to the $L^2$-Wasserstein distance $\W$ on the space $\Pzwei$ of probability measures on $\Rd$ with finite second moment, of the (time-dependent) energy functional
\begin{align}\label{eq:energy}
\mathcal{F}_{t,\omega}(\rho):=\begin{cases}\int_{\Rd}\left[\rho\log(\rho)+\frac12\rho (W_{\omega t}\ast\rho)\right]\dd x&\text{if }m=1,\\
\int_{\Rd}\left[\frac1{m-1}\rho^m+\frac12\rho (W_{\omega t}\ast\rho)\right]\dd x&\text{if }m>1,\end{cases}
\end{align}
provided that the integrals on the r.h.s. are well-defined (otherwise, set $\mathcal{F}_{t,\omega}(\rho):=+\infty$). Note that $\mathcal{F}$ does not possess convexity properties along geodesics in the space $\Pzwei$ \cite{mccann1997}, so the results on contractive gradient flows by Ambrosio, Gigli and Savaré \cite{ags} (in the autonomous case) and Ferreira and Guevara \cite{fg} (in the non-autonomous case) are not immediately applicable.

Still, the cornerstone of our proof is the so-called \emph{Minimizing Movement scheme} \cite{jko} which is by now an almost classical tool to construct weak solutions to equations with gradient flow structure: define, with a given sequence of step sizes $\btau=(\tau_1,\tau_2,\ldots)$, a sequence $(\rho_{\btau,\omega}^k)_{k\in\N}$ recursively by
\begin{align}
			\rho_{\btau,\omega}^k \in \underset{v\in\Pzwei}{\mathrm{argmin}} \left[\frac{1}{2\tau_k}\W^2(\rho^{k-1}_\btau,v) + \mathcal{F}_{t^k_\btau,\omega}(v)\right],\qquad \mathrm{with} \qquad \rho_{\btau,\omega}^0 = \rho_0, \label{MMs}
		\end{align}
and $t^k_\btau=\sum_{l=1}^k \tau_k$ for $k\ge 1$. Further, define the corresponding piecewise constant interpolation $\bar{\rho}_{\btau,\omega}: \left[0,\infty \right)  \rightarrow \Pzwei$ via
	\begin{align}
		\bar{\rho}_{\btau,\omega} (t) = \rho_{\btau,\omega}^k \qquad \mathrm{for} \ t\in \left[ t^k_\btau, t^{k+1}_\btau  \right) \quad \mathrm{and} \ k \in \Nzero. \label{PCInterpolation}
	\end{align}
Our result concerning the limit behaviour as $\btau\to 0$ and $\omega\to\infty$ reads as follows:
\begin{thm}[Existence and high-frequency limit for the nonlinear Fokker-Planck equation]\label{thm:fpe}
Assume that $m\ge 1$ and that $W$ satisfies Assumption \ref{DefRegWintro} and let an initial datum $\rho_0\in \Pzwei$ with $\mathcal{F}_{0,1}(\rho_0)<\infty$ be given.
\begin{enumerate}[(a)]
\item Let $\omega>0$ be fixed. Then, the scheme \eqref{MMs} is well-defined and the following statements hold:
\begin{enumerate}[(i)]
\item For each $T>0$ and each sequence of vanishing step sizes $\btau_n\to 0$, there exists a (non-relabelled) subsequence and a map $\rho_\omega:[0,\infty)\times\Rd\to [0,\infty]$ such that the family of discrete solutions $(\bar{\rho}_{\btau_n,\omega})_{n\in\N}$ converges to $\rho_\omega$ in the following sense:
\begin{align*}
&\bar{\rho}_{\btau_n, \omega}(t,\cdot)\rightharpoonup\rho_\omega(t,\cdot)\text{ narrowly in the space }\Pzwei\text{ at each fixed }t\ge 0,\\
&\bar{\rho}_{\btau_n, \omega}\rightarrow\rho_\omega\text{ in }L^m(0,T;L^m(\Omega))\text{ for every compact set }\Omega\Subset\Rd.
\end{align*}
\item $\rho_\omega$ is a solution to \eqref{FPeq} with the respective $\omega>0$ in the sense of distributions and it has the following properties:
\begin{align*}
\sup_{\omega>0}\left[\|\rho_\omega\|_{\mathcal{C}^{1/2}([0,T],\Pzwei)}+\|\rho_\omega^{m/2}\|_{L^2(0,T;H^1(\Rd))}\right]\le C,
\end{align*}
where the constant $C>0$ depends only on $T$ and $\rho_0$, \emph{but not on $\omega$}.
\end{enumerate}
\item Consider now a sequence $(\omega_n)_{n\in\N}$ with $\omega_n\to\infty$. For every $T>0$, there exists a (non-relabelled) subsequence and a map $\rho_\infty:[0,\infty)\times\Rd\to [0,\infty]$ such that the family $(\rho_{\omega_n})_{n\in\N}$ of weak solutions to \eqref{FPeq} obtained from part (a) converges to $\rho_\infty$,
\begin{align*}
&\rho_{{\omega_n}}(t,\cdot)\rightharpoonup\rho_\infty(t,\cdot)\text{ narrowly in the space }\Pzwei\text{ at each fixed }t\ge 0,\\
&\rho_{{\omega_n}}\rightarrow\rho_\infty\text{ in }L^m(0,T;L^m(\Omega))\text{ for every compact set }\Omega\subset\Rd,
\end{align*}
and $\rho_\infty$ is a solution to the limit problem \eqref{limFPeq} in the sense of distributions.
\end{enumerate}
\end{thm}
The specific setting of nonlinear drift-diffusion equations, seen as gradient flow w.r.t. the $L^2$-Wasserstein distance, is embedded into the more general framework of gradient flows in abstract metric spaces $(\X,\bd)$. There, one considers (time-dependent) energy functionals of the form
\begin{align*}
\mathcal{F}_{(\cdot)}:[0,\infty)\times \X\to \R\cup\{\infty\},\quad \mathcal{F}_t(u)=\ent(u)+\mathcal{P}_t(u),
\end{align*}
consisting of a time-independent part $\ent$ and a time-dependent perturbation $\mathcal{P}_{(\cdot)}$. In this direction, we extend the existence results on non-autonomous gradient flows by Ferreira and Guevara \cite{fg} to energy functionals which are not (semi-)convex along geodesics in the respective metric space. Subsequently, we again consider periodic time-dependent perturbation potentials and study the high-frequency limit of the associated family of curves of steepest descent associated to the gradient system. Our main results read as follows:
\begin{thm}[Curves of steepest descent and high-frequency limit for abstract gradient flows]\label{thm:abstractgf}
Assume that $\ent$ and $\mathcal{P}_t$ fulfill the Assumptions \ref{DefRegE}, \ref{DefRegP} and \ref{DefChainRule}, respectively (see Section \ref{ssec:topolog}). Let an initial datum $u_0\in\X$ with $\mathcal{F}_0(u_0)<\infty$ be given. The following statements hold:
\begin{enumerate}[(a)]
\item There exists a curve $u\in AC^2([0,\infty);\X)$ of steepest descent for $\mathcal{F}$ in $\X$, that is, the following energy balance holds for all $T>0$:
\begin{align*}
			 \mathcal{E} (u(T)) + \mathcal{P}_T (u(T)) + \frac{1}{2} \int_0^T |u' |^2 (t) \dt + \frac{1}{2} \int_0^T | \partial \left(\mathcal{E} + \mathcal{P}_t \right)|^2 (u(t)) \dt  =  \mathcal{E} (u_0) + \mathcal{P}_0 (u_0) + \int_0^T \partial_t \mathcal{P}_t (u(t)) \dt. 
		\end{align*}
\item Assume now in addition that $\mathcal{P}_{(\cdot)}$ is 1-periodic w.r.t. $t$ and that the following regularity conditions hold:
\begin{enumerate}[(1)]
\item $\mathcal{P}$ fulfills a Lipschitz condition in $\X$, uniformly in $t$: there exists $L\ge 0$ such that for all $u,v\in\X$ and all $t\ge 0$, one has, with $\bar{\mathcal{P}}(u):=\int_0^1 \mathcal{P}_t(u)\dd t$, that
$$|(\mathcal{P}_t(u)-\bar{\mathcal{P}}(u))-(\mathcal{P}_t(v)-\bar{\mathcal{P}}(v))|\le L\bd(u,v).$$
\item $\mathcal{P}$ satisfies a convexity condition along geodesics in $\X$: the \emph{Yosida penalization} of $\mathcal{F}$, i.e. the functional
$$\Phi(\tau,t,u;\cdot):=\frac1{2\tau}\bd^2(u,\cdot)+\mathcal{F}_t(\cdot),$$
is (dynamically) $\left(\frac1{\tau}+\lambda(t)\right)$-convex along geodesics in $\X$, for a function $\lambda:[0,\infty)\to\R$ which is bounded from below (see, e.g., \cite{fg}).
\end{enumerate}
Then, the family $(u_\omega)_{\omega>0}$ of curves of steepest descent associated to the family of functionals $\mathcal{F}_{t,\omega}:=\ent+\mathcal{P}_{\omega t}$ for $\omega>0$ from part (a) converges (as $\omega\to\infty$) to a curve $u_\infty$ of steepest descent for the limit functional $\mathcal{F}_\infty:=\ent+\bar{\mathcal{P}}$. Moreover, the following estimate on the convergence rate holds:
\begin{align*}
\bd(u_\omega(t),u_\infty(t))\le\frac{C}{\sqrt{\omega}},\qquad\text{for all }t\in[0,T],
\end{align*}
where the constant $C>0$ depends on $T>0$ and $u_0$. 
\end{enumerate}
\end{thm}
In analogy to the specific case of gradient flows in the space $\Pzwei$ endowed with the $L^2$-Wasserstein distance, the cornerstone of our proofs is the variational approach provided by the so-called \emph{Moreau-Yosida} approximation
\begin{align*}
\phi(\tau,t,u): =\inf_{v\in\X}\Phi(\tau,t,u;v).
\end{align*}
In particular, using the set of all minimizers of the Moreau-Yosida functional $\Phi(\tau,t,u;v)$, denoted by the \emph{resolvent} $J_{\tau,t}(u)$, we are able to define an approximate discrete solution $u_{\btau,\omega}^k \in J_{\tau_k, \omega t^k_{\btau}}(u_{\btau,\omega}^{k-1})$, and the corresponding piecewise constant interpolation $\bar{u}_{\btau,\omega}$ defined via
	\begin{align}
		\bar{u}_{\btau,\omega} (t) =  u_{\btau,\omega}^k \qquad \mathrm{for} \ t\in \left[ t^k_\btau, t^{k+1}_\btau  \right) \quad \mathrm{and} \ k \in \Nzero , \label{PCInterpolation2}
	\end{align}
	which will converge in a suitable notion to a curve of steepest descent. Concerning our existence result Theorem \ref{thm:abstractgf}(a), a similar study has been made by Ferreira and Guevara in \cite{fg}. However, our results hold also without the rather restrictive convexity assumption (1) from Theorem \ref{thm:abstractgf}(b) which has been of extensive use in \cite{fg}. The only concession to be made is that we do not obtain a contractive gradient flow nor uniqueness of curves of steepest descent. Most of our results follow from a careful generalization of the autonomous theory on metric gradient flows by Ambrosio, Gigli and Savaré \cite{ags}, also in view of the theory by Rossi, Mielke and Savaré \cite{rms} for the non-autonomous case under stricter assumptions. Our results on this topic are therefore in close relationship to the results of \cite{ags} and \cite{fg}. The specific novelty of our work lies in to investigation of rapidly oscillating potentials and their behaviour in the limit of infinite oscillation frequency. On grounds of the result in \cite{fg}, we are not only able to show convergence to a time-averaged problem, which is, from a homogenization point of view, the natural limit problem, but obtain a specific rate of convergence (cf. Theorem \ref{thm:abstractgf}(b)). 

As a more concrete application of the abstract results, we study the time-dependent version of the nonlinear Fokker-Planck equation \eqref{FPeq}. Even if the associated free energy functional $\mathcal{F}_{(\cdot)}$ is not (semi-)convex along geodesics, we prove the existence of weak solutions via the well-known Minimizing Movement approximation (see Theorem \ref{thm:fpe}(a)). Since the seminal article by Jordan, Kinderlehrer and Otto on the linear autonomous Fokker-Planck equation \cite{jko}, this method has become fairly classical for the treatment of evolution problems with a formal gradient flow structure (see, for instance, \cite{otto2001, gianazza2009, MatthesMcCannSavare, laurencot2011, blanchet2012, lisini2012, zinsl2014, kmx}). However, in contrast to non-convex confinement potentials (which have already been considered in \cite{jko}), our result on the existence of solutions for problems with non-convex \emph{interaction} potentials is novel (also compare to the works by Carrillo, McCann and Villani \cite{cmv2003, carrillo2006contractions}). Again, the study of highly oscillating potentials is of peculiar interest. The cornerstone of our analysis is the Lipschitz assumption (W6) for the interaction potential which guarantees \emph{a priori} estimates that are uniform with respect to the frequency parameter $\omega>0$. This enables us to pass to the high-frequency limit $\omega\to\infty$ directly in the distributional formulation of equation \eqref{FPeq} to obtain the weak formulation for the limit problem \eqref{limFPeq} (cf. Theorem \ref{thm:fpe}(b)). Numerical studies in one spatial dimension confirming our results have been made in \cite{kim2010}. In contrast to the article by Bartier \emph{et al.}, our techniques \emph{do not rely} on the possibility of rescaling the problem but allow for more general time-dependent interaction potentials.\\

The plan of the paper is as follows. First, we give a brief overview on analysis and gradient flows on metric spaces in Section \ref{Preliminaries} before we set up the framework for time-dependent gradient flows in Section \ref{sec:GradFlow}, proving Theorem \ref{thm:abstractgf}(a). Section \ref{sec:FPeq} is concerned with the nonlinear Fokker-Planck equation \eqref{FPeq}; we prove Theorem \ref{thm:fpe}(a) there. In Section \ref{sec:high}, we first show convergence to the time-averaged abstract problem (completing the proof of Theorem \ref{thm:abstractgf}). Afterwards, the limit behaviour of the nonlinear Fokker-Planck equation \eqref{FPeq} is studied to finish the proof of Theorem \ref{thm:fpe}.


\section{Preliminaries}\label{Preliminaries}

Before starting with the proofs of the main theorems, let us first briefly introduce the theoretical framework of the analysis in abstract metric spaces and later the Wasserstein-formalism of the Fokker-Planck equation. For a broader and more detailed discussion in this direction of the analysis of autonomous problems in metric spaces, we refer to the monograph by Ambrosio \emph{et al.} \cite{ags}. We also want to cite the monograph of Villani \cite{villani} for more details on optimal transportation and the differential structure of the $L^2$-Wasserstein distance as the link between the abstract metric evolution systems and the Fokker-Planck equation.

\subsection{Analysis in metric spaces}

Given a separable, complete metric space $\left(\X, \bd \right)$, we shall introduce a weaker Hausdorff topology $\bs$ on $\bm{X}$, which is compatible with $\bm{d}$, which allows us more flexibility to derive compactness results. From now on we propose the convention to write
	\begin{equation*}
		u_n  \dar u \quad \mathrm{for \ the \ convergence \ w.r.t. \ } \bd, \qquad u_n  \sar u \quad  \mathrm{for \ the \ convergence \ w.r.t. \ } \bs.
	\end{equation*}				
Compatibility of $\bs$ with $\bd$ means in this context
	\begin{equation*}
			u_n  \dar u \ \Longrightarrow 	\ 	u_n  \sar u, \qquad \qquad 			\left(u_n,v_n \right) \sar (u,v) \  \Longrightarrow \ \bd(u,v) \leq \liminf_{\ntoinf} \bd(u_n,v_n).
	\end{equation*}
A curve $u: \left[0,\infty \right) \rightarrow \X$ is said to be $L^2$-absolutely continuous, we write $u\in \mathrm{AC}^2 \left( \left[0,\infty\right),\X\right)$, if there exists a function $m \in L^2_{\mathrm{loc}} \left( \left[0,\infty\right) \right)$ such that
	\begin{equation*}
		\bm{d}(u(t),u(s)) \leq \int_s^t m(r) \dr \qquad \mbox{for all} \ \ 0\leq s \leq t. \label{definitionAbsCont}
	\end{equation*}
Among all possible choices for $m$, there is a minimal one called the metric derivative $|u'|\in L^2_{\mathrm{loc}} \left( \left[0, \infty\right) \right)$, which is explicitly given by
	\begin{equation*}
	|u'|(t) := \lim_{s\rightarrow t} \frac{\bd(u(s),u(t))}{\left|s-t\right|}\qquad \mathrm{for \ a.e. \ } t. \label{MetricDerivative}
	\end{equation*}
Furthermore, sort of a ``modulus of the gradient" for functionals defined on metric spaces is given by the the local slope $| \partial \mathcal{F} |:  \X  \rightarrow \Ru $ of $\mathcal{F}$ at $u \in \X$, defined via
		\begin{equation*}
			| \partial \mathcal{F} |(u) := \limsup_{v \to u} \left(\frac{  \mathcal{F}(u) - \mathcal{F}(v)}{\bd (u,v)}\right)^+ . \label{DefLocalSlope}
		\end{equation*}

\subsection{$L^2$-Wasserstein space}

The underlying space of the Wasserstein-formalism of the Fokker-Planck equation is the space of probability measure $\mu \in \Pzwei$ with bounded second moments, i.e., 
	\begin{equation*}
		\M(\mu) := \int_{\Rd} \left\|x\right\|^2  \dm(x) < \infty .
	\end{equation*}
This space can be endowed with the so-called \emph{$L^2$-Wasserstein distance} $\W$ defined as
\begin{align*}
\W^2(\mu,\nu) := \left( \W(\mu,\nu) \right)^2 := \inf_{\bp \in \Gamma(\mu,\nu)} \ \  \iint\limits_{\Rd\times\Rd} \left\|x-y\right\|^2\, \mathrm{d} \bp(x,y) , 
\end{align*}
where $\Gamma(\mu,\nu)$ denotes the set of all transport plans from $\mu$ to $\nu$, i.e.,
	\begin{equation*}
	\Gamma(\mu,\nu) := \left\{ \bp \in \mathscr{P}(\Rd \times \Rd) : \bp \ \mathrm{ has \ marginals } \ \mu \text{ and } \nu \right\}.
	\end{equation*}
 It is well known that indeed $\left(\Pzwei, \W \right)$ is a separable, complete metric space, see, for instance, \cite[Chapter 6]{ags}. Furthermore, if $\mu$ is absolutely continuous with respect to the Lebesgue measure (write $\mu\in\Paz$), then by the Brenier-McCann theorem \cite[Thm. 2.12]{villani} the infimum in $\W(\mu,\nu)$ is attained at some $\bp \in \Gamma(\mu,\nu)$, which is denoted as the \emph{optimal transport plan}.

Now the weak topology $\bs$ is induced by the $weak^* \ convergence$, or also called \emph{narrow convergence}, of measures, i.e. a sequence $(\mu_n)_{n\in\N}$ is said to \emph{converge narrowly} to some limit probability measure $\mu$ if for all continuous and bounded maps $f:\Rd \to\R$, one has
\begin{align*}
\lim_{\ntoinf}\int_\Rd f(x)\ \mathrm{d} \mu_n (x)&=\int_\Rd f(x) \dm(x).
\end{align*}
The energy functional in the Wasserstein formulation of the Fokker-Planck equation, defined in \eqref{eq:energy}, will be decomposed into a time independent part, the \emph{internal energy}, and into a time-dependent perturbation, the \emph{interaction potential}, i.e., the internal energy is given by
	\begin{align*}
 \mathrm{for} \ m=1: \quad \mathcal{H}(\mu) & :=  \int_\Rd \rho(x) \log \left( \rho(x) \right)  \dx, &	\mathrm{or \ for} \ m>1: \quad \mathcal{U}_m (\mu) & :=  \frac{1}{m-1} \int_\Rd \rho^m(x)  \dx,  
			\end{align*}
	where $\mu$ is absolutely continuous with respect to the Lebesgue measure with density $\rho$. For measures $\mu$, which are singular with respect to the Lebesgue measure, we set $\mathcal{H}(\mu)=\infty$, respectively $\mathcal{U}_m(\mu) = \infty$. Therefore, by a slight abuse of notation, we shall often identify the measure $\mu$ and its corresponding density $\rho$. The according \emph{proper domains} of $\mathcal{H}$ and $\mathcal{U}_m$ are given by
	\begin{align*}	
	 \mathcal{K}_1  := \mathcal{D}\left(\mathcal{H}\right) & =  \left\{ \mu \in \Paz   \left|  \mathrm{d} \mu = \rho \  \mathrm{d} \lambda, \ \rho \log \left( \rho \right)  \in L^1 ( \Rd )  \right.\right\},\\
		\mathcal{K}_m  := \mathcal{D}\left(\mathcal{U}_m\right) & = \left\{ \mu \in \Paz   \left|  \mathrm{d} \mu = \rho \  \mathrm{d} \lambda,\ \rho \in L^m ( \Rd )  \right.\right\} .
	\end{align*}
Further, the interaction potential is defined on the whole space $\Pzwei$ via
		\begin{equation}
			\tW_t(\mu) = \frac{1}{2} \iint\limits_{\Rd \times \Rd} W_t (x,y) \ \mathrm{d} \mu \otimes \mu (x,y). \label{DefInterEng} 
	\end{equation}
	Obviously, the condition $\mathcal{F}_{0,1}(\rho_0)$ for the initial value is then equivalent to $\rho_0 \in \mathcal{K}_m$.

\subsection{Auxiliary theorems}

As a conclusion to this section, we state two theorems which will be of major use to prove our main theorem \eqref{thm:fpe}. First, we state a version of the so-called \emph{flow interchange lemma}  adapted to the situation at hand. For this reason, introduce the $0$-flow $\mathsf{S}^{\mathcal{H}}$ of the entropy functional $\mathcal{H}$ starting at $\mu$, i.e., $\mathsf{S}^{\mathcal{H}}_s(\mu)$ is the unique solution of the evolution variation equation \cite{ags}
	\begin{equation*}
	\frac{1}{2} \frac{\mathrm{d}}{\mathrm{d} t} \left. \W^2( \mathsf{S}^{\mathcal{H}}_t(\mu)  , \nu) \right|_{t=s}	+ \mathcal{H} (\mathsf{S}^{\mathcal{H}}_s(\mu) ) \leq \mathcal{H}(\nu) \qquad  \mathrm{for \ a.e. \ } s>0, \ \text{and all} \ \nu \in \mathcal{K}_1.
	\end{equation*}
	with $\lim_{s\searrow 0} \mathsf{S}^{\mathcal{H}}_s(\mu)  = \mu$ in $\W$. Moreover, $\mathsf{S}^{\mathcal{H}}_s(\mu)$  is also the unique solution of the heat equation with initial value $\mu$, for further details see \cite[Chapter 4]{ags}, and therefore $\mathsf{S}^{\mathcal{H}}_s(\mu)$ inherits all regularization properties of the heat equation.

\begin{thm}[Flow interchange lemma {\cite[Theorem 3.2]{MatthesMcCannSavare}}]\label{FlowInterchange}
Let $\mathcal{V}:\X \rightarrow \R\cup\{\infty\}$ be a proper, lower semi-$\bs$-continuous functional on $(\Pzwei,\W)$ such that $\mathcal{D}(\mathcal{V})\subset \mathcal{K}_1$. Assume that, for arbitrary $\tau>0$ and $\nu \in \Pzwei$, the Moreau-Yosida-functional of $\mathcal{V}$ at $\nu$ possesses a minimizer $\mu$ on $\Pzwei$. Then, the following holds:
\begin{align*}
\mathcal{H}(\mu)+\tau  \mathfrak{D}^\mathcal{H}\mathcal{V}(\mu) &\leq  \mathcal{H}(\nu).
\end{align*}
There, $\mathfrak{D}^\mathcal{V}(\mu)$ denotes the \emph{dissipation} of a functional $\mathcal{V}$ along the $0$-flow $\mathsf{S}^{\mathcal{H}}$ and is given by
\begin{align*}
\mathfrak{D}^\mathcal{H}\mathcal{V}(\mu):=\limsup_{s\searrow 0}\frac{\mathcal{V}(\mu)-\mathcal{V}(\mathsf{S}_{s}^{\mathcal{H}}(\mu))}{s}. 
\end{align*}
\end{thm}
Second, we formulate an extension of the \emph{Aubin-Lions compactness lemma} for metric spaces.
\begin{thm}[Extension of the Aubin-Lions Lemma {\cite[Theorem 2]{RossiSavare}}] \label{LmLmConv}
Let $\X$ be a separable Banach space, $\mathcal{A}:\X \rightarrow \Ru$ be lower semi-continuous and with compact sublevels in $\X$, and $g: \X\times \X \rightarrow \Ru$ be lower semi-continuous and such that $g(u,v)=0$ for $u,v \in \mathcal{D}( \mathcal{A})$ implies $u=v$. Let $\left(u_n\right)_\ninN$ be a sequence of measurable functions $u_n:(0,T)\to\X$ such that
\begin{align}
\sup_{n} & \int_0^T \mathcal{A} \left( u_n(t)\right) \dt < \infty  , &  \liminf_{h \searrow 0 } \limsup_{\ntoinf} & \frac{1}{h} \int_0^h \int_0^{T-t} g(u_n(s+t),u_n(s)) \dss \dt =0 . \label{AssLmLmConv}
\end{align}
Then, $\left(u_n\right)_\ninN$ posses a subsequence converging in measure. 
\end{thm}

\begin{remark}
Note that we replaced the usual weak integral equi-continuity condition 
	\begin{equation*}
		\lim_{h \searrow 0} \sup_{u\in\mathcal{U}}  \int_0^{T-h} g\left( u(t+h),u(t) \right) \ dt =0,
	\end{equation*}
	given in the original version of the theorem, since in the proof of the theorem it is sufficient to have the relaxed averaged weak integral equi-continuity, given in the theorem above.
\end{remark}


\section{Abstract time-dependent gradient flows}\label{sec:GradFlow}


\subsection{Main topological assumptions}\label{ssec:topolog}
	
This chapter of this work is seen as an extension of the well-developed existence theory for autonomous evolution equations in metric spaces \cite{ags}. We shall work throughout the rest of this chapter with the following assumptions to the functional $\mathcal{E}$. 	


	\begin{assumption}[Regularity of $\mathcal{E}$] \label{DefRegE}
	The energy functional $\mathcal{E}: \X \rightarrow \Ru$ is proper and satisfies the following regularity conditions: 
	\begin{enumerate}[({E}1)]
  \item  $\mathcal{E}$ is sequentially lower semi-$\bs$-continuous on $\bd$-bounded sets:
					\begin{equation*}
						\sup_{n,m} \bd (u_n,u_m) < \infty, \quad u_n \sar u \qquad \Longrightarrow \qquad \mathcal{E} (u) \leq \liminf_\ntoinf \mathcal{E}(u_n).
					\end{equation*}
 \item There exist $\tau_* > 0$ and $u_* \in \X$ such that:
					\begin{equation*}
					 c_* := \inf_{v\in \X} \frac{1}{2 \tau_*} \bd^2(u_*,v) + \mathcal{E}(v) > - \infty . 
					\end{equation*}
  \item Every $\bd$-bounded set contained in a sublevel of $\mathcal{E}$ is relatively sequentially $\bs$-compact, i.e.:
					\begin{align*} 
					\begin{split}
						\mathrm{if} & \left(u_n \right)_\ninN \subset  \X \ \mathrm{with} \ \sup_n \mathcal{E}(u_n)< \infty , \ \mathrm{and} \quad \sup_{n,m} \bd(u_n,u_m) < \infty ,  \  \mathrm{then} \\
								&   \ \left(u_n \right)_\ninN  \mathrm{ \ contains \ a \ } \bs \mathrm{ -convergent \ subsequence}. 
					\end{split}
					\end{align*}
					\end{enumerate}
	\end{assumption}

For the analysis of the non-autonomous initial value problem, we need to propose some regularity assumption on the functional $\mathcal{P}_t$ to control the influence of the time-dependent perturbation.


\begin{assumption}[Regularity of $\mathcal{P}_t$] \label{DefRegP}
The perturbation functional $\mathcal{P}_t: \left[0,\infty \right) \times \X \rightarrow \R$ satisfies the following regularity conditions: 
\begin{enumerate}[({P}1)]
\item $\mathcal{P}_t$ is $\bs$-continuous on $\bd$-bounded sets: 
	\begin{equation*}
					\sup_{n,m} \bd(u_n,u_m) < \infty, \quad u_n \sar u, \quad t_n \rightarrow t  \qquad \Longrightarrow \qquad \lim_\ntoinf \mathcal{P}_{t_n}(u_n) = \mathcal{P}_t (u) .  
	\end{equation*}
		\item There exist $\tau^* > 0$ and $u^* \in X$ such that:
					\begin{equation*}
					 c^* := \inf_{t\in \left[ 0, \infty \right) } \inf_{v\in \X} \frac{1}{2 \tau^*} \bd^2(u^*,v) + \mathcal{P}_t(v) > - \infty . 
					\end{equation*}
\item There exists a non-negative function $\alpha \in L^1_{\mathrm{loc}} \left( \left[0,\infty \right) \right)$ such that for all $u \in \X$ and for all $0 \leq s \leq t$, it holds that:
						\begin{equation*}
\left|\mathcal{P}_t(u)-\mathcal{P}_s(u)\right|\leq  \int_s^t \alpha (r) \dr (1+ \bd^2(u^*,u)). 
					\end{equation*}
					Moreover, the set $\mathcal{N}_\alpha$ given by 
					\begin{equation*}
						\mathcal{N}_\alpha := \left\{ t>0 \left| \liminf_{\tau \searrow 0} \sup_{\sigma \in \left(0,\tau \right) } \frac{\sigma}{\tau-\sigma} \int_{t+ \sigma}^{t+ \tau} \alpha(r) \dr = \infty \right. \right\}
					\end{equation*}
					is at most countable.
		\item For all $u \in \X$, the partial derivative $\partial_t \mathcal{P}_t(u)$ exists and is $\bs$-continuous on $\bd$-bounded sets:
	\begin{equation*}
					\sup_{n,m} \bd(u_n,u_m) < \infty, \quad u_n \sar u, \quad t_n \rightarrow t \qquad \Longrightarrow \qquad \lim_\ntoinf \partial_t \mathcal{P}_{t_n}(u_n) = \partial_t \mathcal{P}_t (u) .  \label{DiffofP}
	\end{equation*}
	\end{enumerate}
\end{assumption}

\begin{remark} Without loss of generality, we can assume that $\tau_* = \tau^*$, $c_*=c^*$, and $u_*=u^*$ and that
	\begin{equation*}
					 c_* = \inf_{t\in \left[ 0, \infty \right) } \inf_{v\in \X} \frac{1}{2 \tau_*} \bd^2(u_*,v) + \mathcal{F} (v)  + \mathcal{P}_t(v) > - \infty . 
	\end{equation*}
	Further, note that $\mathcal{N}_\alpha$ is at most countable if $\alpha$ is of bounded mean oscillation, essentially bounded or has at most countable many points which are not Lebesgue points. 
\end{remark}

As already mentioned in \cite{rms,fg} a crucial ingredient in the derivation of the energy identity for curves of steepest descent is the chain rule inequality.


\begin{assumption}[Local slope and chain rule inequality]\label{DefChainRule}
The local slope $|\partial \left( \mathcal{E} + \mathcal{P}_t \right) |$ of $\mathcal{E} + \mathcal{P}_t$ at time $t$ is lower semi-$\bs$-continuous and satisfies the chain rule condition, i.e., if for any curve $u\in \mathrm{AC}^2 \left( \left[0,\infty\right),\X\right)$ with 
	\begin{align*}
		|\partial \left( \mathcal{E} + \mathcal{P}_t \right)  |(u(t)) |u'|(t) \in L^1_{\mathrm{loc}} \left( \left[0,\infty\right)\right) \qquad \mathrm{and} \qquad 
		\sup_{t\in \left[0,T\right]} ( \mathcal{E} + \mathcal{P}_t ) (u(t)) < \infty,
	\end{align*}
the map $t \mapsto \mathcal{E}(u(t)) + \mathcal{P}_t(u(t))$ is absolutely continuous, and for all $ 0 \leq s \leq t$:
	\begin{equation}
		 \mathcal{E}(u(s)) + \mathcal{P}_s(u(s))+ \int_s^t \partial_t \mathcal{P}_r ( u(r)) \dr \leq \mathcal{E}(u(t)) + \mathcal{P}_t(u(t)) + \int_s^t |\partial\left( \mathcal{E} + \mathcal{P}_t \right)  |(u(r)) |u'|(r) \dr.   \label{UpperGradientIneq}
	\end{equation}
\end{assumption}


\subsection{Moreau-Yosida approximation and resolvent}

Define the \emph{Moreau-Yosida functional}
				\begin{equation*}
					\Phi(\tau,t,u; \cdot):  \X \rightarrow \Ru; \ \Phi(\tau,t,u; v):= \frac{1}{2\tau} \bd^2(u,v) + \mathcal{E}(v) + \mathcal{P}_t(v) \label{MYfct}
				\end{equation*}	
			and furthermore define the \emph{Moreau-Yosida approximation} of $\mathcal{E}+\mathcal{P}_t$ by
							\begin{equation*}
					\phi(\tau,t,u):=\inf_{v \in \X} \Phi(\tau,t,u; v) = \inf_{v \in \X} \frac{1}{2\tau} \bd^2(u,v) + \mathcal{E}(v) + \mathcal{P}_t(v). \label{MYApprox}
				\end{equation*}	
	The well-posedness of the \emph{Minimizing Movement scheme} \eqref{MMs} is equivalent to the existence of a minimizer of the $\phi$. The set of all minimizers is called the \emph{resolvent} $J_{\tau,t}$ and is given by
		\begin{equation*}
			J_{\tau,t}(u) = \left\{ v \in \X \mid \Phi(\tau,t,u; v)  =  \phi(\tau,t,u)  \right\}. \label{Resolvent}
		\end{equation*}

\begin{remark} 
By construction of the Moreau-Yosida approximation, we have the following monotonicity 
		\begin{equation}
			 \phi(\sigma,t,u) 	\leq \phi(\tau,t,u) \leq	\mathcal{E}(u) + \mathcal{P}_t(u)  \qquad \mbox{for} \ \sigma \geq \tau >0. \label{MYmonotonicity}
		\end{equation}
\end{remark}

In the following we show that indeed the \emph{time-dependent Minimizing Movement scheme} \eqref{MMs} is well-defined in the abstract metric setting for sufficiently small $\tau$. Further facts about the time-dependent facts of the Moreau-Yosida approximation and of the resolvent \cite{ags}, for instance a priori estimates, continuity results and differentiability properties, in the case if the functionals are not bounded from below, as in \cite{rms}, or do not exhibit a $\lambda$-convex structure, as in \cite{fg}, are skipped for the sake of brevity and given in the Appendix.


\begin{thm}[Existence of a minimizer] \label{ExistenceMMS}
	For all $\tau \in \left(0 , \tau_* \right)$, for all $t \in \left[0,\infty \right)$ and for all $ u \in \X$, there exists a minimizer $v_* \in \mathcal{D}(\mathcal{E})$ of $\ \Phi(\tau,t,u, \cdot)$, i.e., 
		\begin{equation*}
			J_{\tau,t}(u) \neq \emptyset .
		\end{equation*}
\end{thm}

	\begin{proof}
		Fix $ \tau \in \left(0, \tau_* \right), t \in \left[0, \infty \right), u \in \X$ and note that by Lemma \ref{BoundsMY} the Moreau-Yosida functional is bounded from below for each $u \in \X$. Since $\mathcal{E}$ and $\mathcal{P}_t$ are proper, the infimum is not equal to infinity. So choose a minimizing sequence $\left(v_n \right)_\ninN$ in $\X$ of $\Phi(\tau,t,u; \cdot)$ and without loss of generality  $\sup_n \Phi(\tau,t,u; v_n) < \infty $. So, we can deduce from  \eqref{UBoundMY} 
		\begin{align*}
				\bd^2(v_n,u) & \leq \frac{4\tau \tau_*}{\tau_* - \tau} \left( \Phi(\tau,t,u;v_n) - c_*  + \frac{1}{\tau_* - \tau} \bd^2(u_*,u) \right) < \infty.
		\end{align*}
		Thus the sequence $\left(v_n\right)_\ninN$ is $\bd$-bounded. Furthermore, the $\bs$-compactness of the sequence $v_n$ follows by the upper estimate on $\mathcal{E}$
				\begin{align*}
			\mathcal{E}(v_n) & \leq \frac{1}{2\tau}\bd^2(u,v_n)  + \mathcal{E}(v_n) + \mathcal{P}_t(v_n) - c^* =  \Phi(\tau,t,u;v_n) - c^* \leq c  < \infty.
		\end{align*}
		Hence, we can extract a $\bs$-convergent subsequence, which converges to some $v_* \in \mathcal{D}(\mathcal{E})$ with respect to the weak topology $\bs$. By lower semi-$\bs$-continuity of $\mathcal{E}$ and $\bs$-continuity of $\mathcal{P}_t$, we conclude that indeed $v_*$ is a minimizer of $\Phi(\tau,t,u; \cdot)$ and thus $J_{\tau,t}(u) \neq \emptyset $.	
	\end{proof}


\subsection{Minimizing Movement scheme}
In this section we prove that the approximate solution given via the \emph{Minimizing Movement scheme} converges to a curve of steepest descent of the functional $\mathcal{E} + \mathcal{P}_t$. To do so, we establish some classical estimates for the \emph{piecewise constant interpolation}, which guarantees the convergence at least for a subsequence.


\begin{thm}[Classical estimates I] \label{BoundsMMS}
Let $u_0 \in \mathcal{D}\left(\mathcal{E}\right)$. For fixed $T>0$, there exists a constant $C(T,\tau_*,u_0)$, only depending on $T, \tau_*$ and $u_0$, such that for all partitions  $\mathcal{T}$ with $\btau$ sufficiently small, i.e.,
	\begin{equation*}
	\sup_k 4 \alpha^k_\btau  < 1 \qquad \mathrm{with} \qquad \alpha^k_\btau:=  \left( \frac{\tau_*}{2} \int_{t^k_\btau}^{t^{k+1}_\btau} \alpha (r) \dr   +\frac{ \tau_k}{\tau_*} \right),
	\end{equation*}
such that the corresponding discrete solutions $\left( \uk \right)_\kinN$ satisfy for all $N$ with $t^N_\btau < T$:
		\begin{align}
		\sum_{k=1}^N \frac{1}{2\tau_k} \dk & \leq  C(T,\tau_*,u_0) , \label{boundsumdk} \\
		\mathcal{E}(\ukn) & \leq C(T,\tau_*,u_0), \label{BoundFk} \\
		\bd^2(u_*,\ukn) & \leq C(T,\tau_*,u_0) \label{bounddn}.
		\end{align}
\end{thm}


The proof of these bounds can be found in the appendix. The convergence of the piecewise constant interpolation is now guaranteed by these bounds and a refined version of the Arzelà-Ascoli theorem, which can be found in \cite[Proposition 3.3.1]{ags}.

\begin{thm}[Convergence of the piecewise constant interpolation] \label{ExistenceMMSLimitcurvePC}
Given a family of partitions $\left(\mathcal{T}_n \right)_\ninN$, with $\btau_n$ sufficiently small and $\btau_n \searrow 0$. For $u_0 \in \mathcal{D}(\mathcal{E})$ define the corresponding piecewise constant interpolation $\bar{u}_{\btau_n}$. Then there exists $u_* \in \mathrm{AC}^2 \left(\left[0, \infty \right), \X \right)$ such that for a non-relabelled subsequence of $\btau_n$ 
		\begin{equation*}
			\bar{u}_{\btau_n} (t) \sar u_* (t) \qquad \forall \ t\in \left[0,\infty \right).
		\end{equation*}
		\end{thm}

\begin{proof}
Fix some $T>0$ and define the discrete derivative $ | \bar{u}_{\btau_n}' |$ as
	\begin{equation}
		 | \bar{u}_{\btau_n} ' |(t):= \frac{1}{\tau_{n,k}} \bd(u^k_{\btau_n},u^{k+1}_{\btau_n}) \qquad \mathrm{for} \quad t \in \left[ t^k_{\btau_n}, t^{k+1}_{\btau_n}\right). \label{discrete derivative}
	\end{equation}
Using the classical estimates for the Minimizing Movement scheme of Theorem \ref{BoundsMMS},  we get for all $t^N_{\btau_n} <T$
	\begin{align*}
		\int_0^{t^N_{\btau_n}} | \bar{u}_{\btau_n} ' |^2 (t) \dt  = \sum_{k=0}^{N-1} \int_{t^k_{\btau_n}}^{t^{k+1}_{\btau_n}} \left( \frac{1}{\tau_{n,k}} \bd(u^k_{\btau_n},u^{k+1}_{\btau_n})  \right)^2 \dt =  \sum_{k=1}^N  \frac{1}{\tau_{n,k}} \bd^2(u^k_{\btau_n},u^{k-1}_{\btau_n})  \leq 2 C(T).
	\end{align*}
Thus $ | \bar{u}_{\btau_n} ' | \in L^2\left(0,T\right)$ and the ${L^2\left(0,T\right)}$-norm of $ | \bar{u}_{\btau_n} ' |$ is  uniformly bounded in $\btau_n$. Therefore $| \bar{u}_{\btau_n}' |$ possesses a $L^2\left(0,T\right)$-weakly convergent non-relabelled subsequence with limit $A \in L^2\left(0,T\right)$. Choose $0 \leq s \leq t \leq T$ arbitrary and define $k_n(t):= \max \left\{ k \mid t^k_{\btau_n} \leq t \right\}$, then
	\begin{align*}
 \bd(\bar{u}_{\btau_n}(s),\bar{u}_{\btau_n}(t))  \leq \sum_{k=k_n(s) +1}^{k_n(t)} \bd(u_{\btau_n}^k,u_{\btau_n}^{k-1}) = \sum_{k=k_n(s) +1}^{k_n(t)} \int_{t^{k-1}_{\btau_n}}^{t^k_{\btau_n}} \frac{1}{\tau_{n,k}} \bd(u_{\btau_n}^k,u_{\btau_n}^{k-1}) \dr  = \int_{t^{k_n(s)}_{\btau_n}}^{t^{k_n(t)}_{\btau_n}} | \bar{u}_{\btau_n}' |(r) \dr.
	\end{align*}
Taking the limit $\ntoinf$ yields now 
	\begin{equation*}
		\limsup_{k \rightarrow \infty} \bd(\bar{u}_{\btau_n}(s),\bar{u}_{\btau_n}(t)) \leq \int_s^t A(r) \dr.
	\end{equation*}
On the other hand, by the classical estimate \eqref{BoundFk} for the Minimizing Movement scheme, the piecewise constant interpolation $\bar{u}_{\btau_n} (t)$ is contained in some sublevel of $\mathcal{E}$ uniformly in $t\in \left[0,T\right]$ and for sufficiently small $\btau$. Estimate \eqref{bounddn} additionally ensures the uniform $\bd$-boundedness of $\bar{u}_{\btau_n} (t)$ and therefore, using the $\bs$-compactness of $\mathcal{E}$, $\bar{u}_{\btau_n} (t)$ is contained in some $\bs$-compact set $K$ for all $t\in \left[0,T\right]$ and for all $\ninN$.

Hence we can apply the refined Arzelà-Ascoli Theorem \cite[Proposition 3.3.1]{ags} yielding the existence of a non-relabelled subsequence and a limit curve $u_*: \left[0,T \right] \rightarrow \X$ such that $\bar{u}_{\btau_n} (t) \sar u_*(t)$ for each fixed $t\in\left[0,T\right]$. Moreover, $u_*$ is absolutely continuous and consequently, a diagonal argument yields convergence w.r.t. $\bs$ on $\left[0,\infty\right)$.
\end{proof}


	 Now, we define a refined interpolation, the so called \emph{De Giorgi interpolation} $\tilde{u}_{\tau}: \left[0,\infty \right) \rightarrow \X$, via
	\begin{equation*}
		\tilde{u}_{\btau} (t^k_\btau ) = \uk, \qquad \tilde{u}_{\btau} (t^k_\btau +\sigma) \in J_{\sigma, t^k_\btau +\sigma }(\uk) \qquad \mathrm{for} \ \sigma \in \left(0,\tau_{k+1} \right) , \ \kinN_0,
	\end{equation*}
	 which satisfies the following \emph{discrete energy inequality}.  

\begin{thm}[Discrete energy inequality]\label{DiscIntInequMMs}
 For $u_0 \in \mathcal{D}\left(\mathcal{E}\right)$ and a given partition $\mathcal{T}$ with $\btau\in \left( 0 , \tau_* \right)$, let  $ \left( \uk \right)_\kinN$ be the discrete solution. Then the De Giorgi interpolation $\tilde{u}_\btau(t)$ satisfies:
		\begin{align}\label{DiscIntInequMMsInequ} 
			\begin{split}
& \mathcal{E}(\tilde{u}_\btau(t^N_\btau)) + \mathcal{P}_{t^N_\btau} (\tilde{u}_\btau(t^N_\btau))  +	\sum_{k=1}^{N } \frac{1}{2\tau_k} \dk + \frac{1}{2}  \int_0^{t^N_\btau} |  \partial\left( \mathcal{E} + \mathcal{P}_{t} \right) |^2 ( \tilde{u}_\btau (t) ) \dt  \\
 \leq & \mathcal{E}(u_0) + \mathcal{P}_{0} (u_0) + \int_0^{t^N_\btau} \partial_t \mathcal{P}_{t}(\tilde{u}_\btau (t)) \dt. \end{split}
	\end{align}

\end{thm}

This proof is similar to the proof in \cite{fg} and can be found in the Appendix for the sake of completeness. Next we prove that also the De Giorgi interpolation $\tilde{u}_{\btau_n}$ converges pointwise with respect to the weak topology $\bs$ to some limit curve $\tilde{u}_*(t)$. In the end, we identify the new limit curve with limit curve $u_*$ obtained previously:

\begin{thm}[Convergence of the De Giorgi interpolation] \label{ExistenceMMSLimitcurveDG}
Given a family of partitions $\left(\mathcal{T}_n \right)_\ninN$, with $\btau_n$ sufficiently small and $\btau_n \searrow 0$. For $u_0 \in \mathcal{D}(\mathcal{E})$ define the corresponding De Giorgi interpolation  $\tilde{u}_{\btau_n}$ and let $u_*(t)$ be the limit curve of the piecewise constant interpolation $\bar{u}_n$, then there exists a non-relabelled subsequence of $\tilde{u}_{\btau_n}$such that
		\begin{equation}
\tilde{u}_{\btau_n} (t) \sar u_* (t) \qquad \forall \ t \in \left[0,\infty \right)\backslash \mathcal{N}_\alpha . \label{ExistenceMMSLimitDG}
		\end{equation}
		\end{thm}


\begin{proof}


 As before, we prove that the family of De Giorgi interpolations $\tilde{u}_{\btau_n}$ is contained in some $\bs$-compact set $\tilde{K}$. The $\bd$-boundedness of $\tilde{u}_{\btau_n}$ follows by \eqref{UBoundMY}, since for $t= t^{k_n(t)}_{\btau_n} + \sigma$
\begin{align*}
	\bd^2(\tilde{u}_{\btau_n}  (t) ,\bar{u}_{\btau_n} (t) )= \bd^2(\tilde{u}_{\btau_n}  (t) , u_{\btau_n}^{k_n(t)} )& \leq \frac{4\sigma \tau_*}{\tau_* - \sigma } \left( \Phi(\sigma,t,u_{\btau_n}^{k_n(t)},\tilde{u}_{\btau_n} (t)) - c_*  + \frac{1}{\tau_* - \sigma} \bd^2(u_*,u_{\btau_n}^{k_n(t)}) \right) \\
	& = \frac{4\sigma \tau_*}{\tau_* - \sigma } \left( \phi(\sigma,t,u_{\btau_n}^{k_n(t)}) - c_*  + \frac{1}{\tau_* - \sigma} \bd^2(u_*,u_{\btau_n}^{k_n(t)}) \right) \\
	& \leq \frac{4\sigma \tau_*}{\tau_* - \sigma } \left( \mathcal{E}( u_{\btau_n}^{k_n(t)}) + \mathcal{P}_t(u_{\btau_n}^{k_n(t)}) - c_*  + \frac{1}{\tau_* - \sigma} \bd^2(u_*,u_{\btau_n}^{k_n(t)}) \right).
\end{align*}
The first term on the right hand side is bounded by the constant given in Theorem \ref{BoundsMMS}. Since the discrete solution $u_{\btau_n}^k$ is locally contained in the $\bs$-compact set $K$, the second term is also bounded locally in time by some constant independent of $t$ and $\tau$. The third term is bounded by the classical estimates \eqref{bounddn}, hence the De Giorgi interpolation $\tilde{u}_{\btau_n}$ is locally $\bd$-bounded. 

Next we prove the boundedness of $\mathcal{E}(\tilde{u}_{\btau_n}(t))$, to do so, use the coercivity of $\mathcal{P}_t$ to obtain
	\begin{align*}
	\mathcal{E}(\tilde{u}_{\btau_n} (t) ) & \leq \frac{1}{2\sigma} \bd^2(u^{k_n(t)}_{\btau_n},\tilde{u}_{\btau_n} (t) ) + \mathcal{E}(\tilde{u}_{\btau_n} (t) ) + \frac{1}{2\tau_*} \bd^2(u_*,\tilde{u}_{\btau_n} (t)) + \mathcal{P}_{t}(\tilde{u}_{\btau_n} (t) ) - c_* \\
	& = \phi(\sigma, t , u^{k_n(t)}_{\btau_n} , \tilde{u}_{\btau_n} (t) ) + \frac{1}{2\tau_*} \bd^2(u_*,\tilde{u}_{\btau_n} (t)) - c_*\\
	& \leq \mathcal{E}(u^{k_n(t)}_{\btau_n}) + \mathcal{P}_t(u^{k_n(t)}_{\btau_n}) + \frac{1}{2\tau_*} \bd^2(u_*,\tilde{u}_{\btau_n} (t))   -c_* .
	\end{align*}
Again, the first two terms are bounded by the previous argument and the third term is bounded as shown before, hence the De Giorgi interpolations $\tilde{u}_{\btau_n}(t)$ are contained in a $\bs$-compact set $\tilde{K}$. 

To apply the refined Arzelà-Ascoli theorem \cite[Proposition 3.3.1]{ags}, it remains to prove an estimate in terms of the modulus of continuity. Note that by Lemma \ref{EstimateMinimizers}, one has 
	\begin{align*}
	 \bd^2(\bar{u}_{\btau_n}(t) , \tilde{u}_{\btau_n}(t) ) & \leq \bd^2(u^{k_n(t)}_{\btau_n}, u^{k_n(t)+1}_{\btau_n}) + 2\tau_{n,k}  \frac{t- t^{k_n(t)}_{\btau_n}}{t^{k_n(t)+1}_{\btau_n}-t} \int_t^{t^{k_n(t)+1}_{\btau_n}} \alpha(r) \dr \left( 2 +\bd^2( u_*,\tilde{u}_{\btau_n}(t)) + \bd^2( u_*,u^{k_n(t)}_{\btau_n}) \right) \\
	& \leq \btau_n C(T) \left( 1 +  \frac{t- t^{k_n(t)}_{\btau_n}}{t^{k_n(t)+1}_{\btau_n}-t} \int_t^{t^{k_n(t)+1}_{\btau_n}} \alpha(r) \dr \right).
	\end{align*}
	Here, we estimated the first term by \eqref{boundsumdk} and the last term using the $\bd$-boundedness of $\bar{u}_{\btau_n} (t)$ and $\tilde{u}_{\btau_n}(t)$. By the assumption on $\mathcal{P}_t$, the last term is bounded locally in time, uniformly in $\btau$, except on an at most countable set $\mathcal{N}_\alpha$. Therefore, we can deduce for $s,t \in \left[0,T\right] \backslash \mathcal{N}_\alpha$
\begin{align*}
		\limsup_{\ntoinf} \bd^2( \tilde{u}_{\btau_n} (t) , \tilde{u}_{\btau_n} (s) ) & \leq \limsup_{\ntoinf} 3 \left( \bd^2( \tilde{u}_{\btau_n}  (t) , \bar{u}_{\btau_n}  (t) ) + \bd^2( \bar{u}_{\btau_n}  (t) , \bar{u}_{\btau_n}  (s) ) +\bd^2( \bar{u}_{\btau_n}  (s) , \tilde{u}_{\btau_n}  (s) ) \right) \\
		& \leq  \limsup_{\ntoinf} 6\btau_n C(T)+ \limsup_{\ntoinf} 3  \bd^2( \bar{u}_{\btau_n}  (t) , \bar{u}_{\btau_n}  (s) )   \leq 3  \int_s^t A(r) \dr.
\end{align*}
	Hence, we can apply the refined version of the Arzelà-Ascoli theorem \cite[Proposition 3.3.1]{ags}, to conclude the pointwise convergence on $\left[0,T\right]$ with respect to the topology $\bs$ of the De Giorgi interpolation $\tilde{u}_\tau(t)$ to an absolutely continuous curve $\tilde{u}_*(t)$ for a non-relabelled subsequence. A diagonal argument yields the convergence with respect to $\bs$ on $\left[0,\infty\right)$ for a further non-relabelled subsequence. In particular, the two limit curves $u_*$ and $\tilde u_*$ have to agree since by the compatibility assumption to the weak topology $\bs$ one has at least on the set $\left[0, \infty \right) \backslash \mathcal{N}_\alpha$
	\begin{equation*}
 		\bd^2(\tilde{u}_*(t),u_*(t)) \leq \liminf_{\ktoinf} \bd^2(\tilde{u}_{n_k}(t),\bar{u}_{n_k}(t)) \leq \liminf_{k\to\infty}  \btau_{n_k}C(T)  =0 . \qedhere
\end{equation*}
\end{proof}


\subsection{Existence of curves of steepest descent}


Finally, we are able to complete the proof of Theorem \ref{thm:abstractgf}(a), i.e., the existence of \emph{curves of steepest descent}:


\noindent\emph{Proof of Theorem \ref{thm:abstractgf}(a)}
In the sequel, we prove that the limit curve $u_* (t)\in \mathrm{AC}^2 \left(\left[0, \infty \right) ,\X \right)$ given by the Minimizing Movement scheme starting at $u_0$ using the family $\mathcal{T}_n$ of feasible partitions --- without loss of generality $\bar{u}_{\btau_n} \sar u_*$ and $\tilde{u}_{\btau_n} \sar u_*$ --- is a curve of steepest descent for the functional $\mathcal{E} + \mathcal{P}_t$. We know that the De Giorgi interpolation satisfies the discrete energy inequality \eqref{DiscIntInequMMsInequ}, i.e.,
\begin{align*} 
\begin{split}
& \mathcal{E}(\tilde{u}_{\btau_n}(t^N_{\btau_n})) + \mathcal{P}_{t^N_{\btau_n}} (\tilde{u}_{\btau_n}(t^N_{\btau_n}))  +	\sum_{k=1}^{N } \frac{1}{2\tau_{n,k}} \bd^2( u^k_{\btau_n}, u^{k+1}_{\btau_n}) + \frac{1}{2}  \int_0^{t^N_\btau} |  \partial\left( \mathcal{E} + \mathcal{P}_{t} \right) |^2 ( \tilde{u}_{\btau_n} (t) ) \dt \\
 \leq & \mathcal{E}(u_0) + \mathcal{P}_{0} (u_0) + \int_0^{t^N_\btau} \partial_t \mathcal{P}_{t}(\tilde{u}_{\btau_n} (t)) \dt.
\end{split}
	\end{align*}
Fix $T \in \left[0,\infty \right)$ and compute the limes inferior of the l.h.s. of the equation above. Since $\tilde{u}_{\btau_n}(t^{k_n(T)}_{\btau_n}) = \bar{u}_{\btau_n}(T) \sar u_* (T)$, we have by the lower semi-$\bs$-continuity of $\mathcal{E}$ and the $\bs$-continuity of $\mathcal{P}_t$:
	\begin{equation*}
		\mathcal{E} (u_*(T)) + \mathcal{P}_t ( u_* (T))	\leq\liminf_{\ntoinf} \left(\mathcal{E} ( \bar{u}_{\btau_n}(T)) + \mathcal{P}_{t^N_{\btau_n}} ( \bar{u}_{\btau_n}(T))\right)  =  \liminf_{\ntoinf} \left(\mathcal{E}(\tilde{u}_{\btau_n}(t^{k_n(T)}_{\btau_n})) + \mathcal{P}_{t^N_{\btau_n}} (\tilde{u}_{\btau_n}(t^{k_n(T)}_{\btau_n}))\right).
	\end{equation*}
	In the proof of Theorem \ref{ExistenceMMSLimitcurvePC}, we have seen that the discrete derivative $| \bar{u}_{\btau_n} '| (t)$ converges weakly to $A(t)$ in $L^2\left(0,T\right)$, with $A(t)$ is one possible modulus of continuity in the definition of absolute continuity. Furthermore, since the metric derivative $ | u_* '| (t) $ is the smallest modulus of continuity, one has $| u_* '| (t) \leq A(t)$ almost everywhere. The weak lower semi-continuity of the $L^2 \left(0,T\right)$-norm implies then:
	\begin{equation*}
	\frac{1}{2}  \int_0^T |u_* ' |^2 (t) \dt \leq  \frac{1}{2} \int_0^T A(t)^2 \dt \leq \liminf_{\ntoinf} \frac{1}{2} \int_0^{t^{k_n(T)}_{\btau_n}} | \bar{u}_{\btau_n} '|^2 (t) \dt = \liminf_{\ntoinf} \sum_{k=1}^{k_n(T)} \frac{1}{2\tau_{n,k}} \bd^2( u^k_{\btau_n}, u^{k+1}_{\btau_n}) .
	\end{equation*}	
	The De Giorgi interpolation converges weakly almost everywhere in $t$, so using Fatou's lemma and the lower semi-$\bs$-continuity of the local slope $| \partial (\mathcal{E} + \mathcal{P}_t )|$ yields for the last term on the l.h.s.
	\begin{align*}
		\frac{1}{2} \int_0^T | \partial \left(\mathcal{E} + \mathcal{P}_t \right)|^2 (u_*(t)) \dt  \leq 		\frac{1}{2} \int_0^T \liminf_{\ntoinf} | \partial \left(\mathcal{E} + \mathcal{P}_t \right)|^2 (\tilde{u}_{\btau_n} (t)) \dt  \leq \liminf_{\ntoinf} \frac{1}{2} \int_0^{t_{\btau_n}^{k_n(T)}} | \partial (\mathcal{E} + \mathcal{P}_{t} )|^2 ( \tilde{u}_{\btau_n}  (t) ) \dt.
	\end{align*}
		At last, we compute the limit of the right-hand side by applying the dominated convergence theorem. Clearly, we have pointwise convergence of $\partial_t \mathcal{P}_t (\tilde{u}_{\btau_n}(t))$ to $\partial_t \mathcal{P}_t (u_*(t))$ almost everywhere. Since the De Giorgi interpolation is locally contained in the $\bs$-compact set $\tilde{K}$ and $\partial_t \mathcal{P}_t$ is $\bs$-continuous, the integrand is uniformly bounded by some constant. Hence we can conclude with the dominated convergence theorem that
	\begin{equation*}
	 \int_0^T \partial_t \mathcal{P}_t (u_*(t)) \dt = \lim_{\ntoinf}  \int_0^{ t_{\btau_n}^{k_n(T)} } \partial_t \mathcal{P}_{t}(\tilde{u}_{\btau_n}(t)) \dt.
	\end{equation*}
	Summarized, we have the following energy inequality:
			\begin{align*}
			\mathcal{E} (u_*(T)) + \mathcal{P}_T (u_*(T)) + \frac{1}{2} \int_0^T |u_* ' |^2 (t) \dt + \frac{1}{2} \int_0^T | \partial \left(\mathcal{E} + \mathcal{P}_t \right)|^2 (u_*(t)) \dt \leq  \mathcal{E} (u_0) + \mathcal{P}_0 (u_0) + \int_0^T \partial_t \mathcal{P}_t (u_*(t)) \dt. 
		\end{align*}	
	The reversed inequality follows now by the chain rule assumption, since the energy inequality above yields an upper estimate for the $L^1$-norm of $| \partial (\mathcal{E} + \mathcal{P}_t) | (u_*(t)) |u_*'|(t)$, i.e.,
	\begin{align*}
		\int_0^T | \partial (\mathcal{E} + \mathcal{P}_t) | (u(t)) |u'|(t) \dt & \leq \frac{1}{2} \int_0^T |u_* ' |^2 (t) \dt + \frac{1}{2} \int_0^T | \partial \left(\mathcal{E} + \mathcal{P}_t \right)|^2 (u_*(t)) \dt \notag \\
		& \leq \mathcal{E} (u_0) + \mathcal{P}_0 (u_0) - \mathcal{E} (u_*(T)) - \mathcal{P}_T (u_*(T)) + \int_0^T \partial_t \mathcal{P}_t (u_*(t)) \dt . 
	\end{align*}
	The right-hand-side is is always finite and therefore $| \partial (\mathcal{E} + \mathcal{P}_t) | (u_*(t)) |u_*'|(t) \in L^1_{\mathrm{loc}} \left( \left[0, \infty \right) \right)$. The boundedness from above of $\mathcal{E} (u_*(t)) + \mathcal{P}_t(u_*(t)) $ is given by 
	\begin{equation*}
	\mathcal{E} (u_*(T)) + \mathcal{P}_t ( u_* (T))	\leq\liminf_{\ntoinf} \left(\mathcal{E} ( \bar{u}_{\btau_n}(T)) + \mathcal{P}_{t^N_{\btau_n}} ( \bar{u}_{\btau_n}(T))\right).
	\end{equation*}
	Once more, the first term is bounded from above by the classical estimate \eqref{BoundFk} and the second term is bounded due to the $\bs$-continuity of $\mathcal{P}_t$ and the $\bd$-compactness of $u_* ( \left[0, T\right])$.	Thus we can apply the chain rule inequality \eqref{UpperGradientIneq} to obtain
		\begin{align*}
		\int_0^T | \partial (\mathcal{E} + \mathcal{P}_t) | (u(t)) |u'|(t) \dt & \leq \frac{1}{2} \int_0^T |u_* ' |^2 (t) \dt + \frac{1}{2} \int_0^T | \partial \left(\mathcal{E} + \mathcal{P}_t \right)|^2 (u_*(t)) \dt  \\
		& \leq \mathcal{E} (u_0) + \mathcal{P}_0 (u_0) - \mathcal{E} (u_*(T)) - \mathcal{P}_T (u_*(T)) + \int_0^T \partial_t \mathcal{P}_t (u_*(t)) \dt  \\
		& \leq \int_0^T | \partial (\mathcal{E} + \mathcal{P}_t) | (u(t)) |u'|(t) \dt .
	\end{align*}
	Therefore all inequalities have to be equalities and $u_*(t)$ is in fact a curve of steepest descent for the functional $\mathcal{E}+ \mathcal{P}_t$. 	
\qed


\section{Fokker-Planck equation}\label{sec:FPeq}

In this section, we prove that under suitable regularity assumptions, there exist weak solutions of the linear and non-linear Fokker-Planck equation \eqref{FPeq}. In contrast to \cite{fg}, we prove the existence even when the interaction energy and therefore the Moreau-Yosida functional is not $\lambda$-convex (for any $\lambda\in\R$). Our strategy of proof is as follows. We use the Minimizing movement scheme to construct an approximate weak solution to equation \eqref{FPeq}. Due to the nonlinearity of the equation at hand, however, the classical estimates provided by the scheme (cf. Theorem \ref{BoundsMMS}) are not strong enough to pass to the limit in the approximate weak formulation. Therefore, additional regularity estimates are derived using the flow interchange technique. Note that the regularity conditions of Assumption \ref{DefRegWintro} are divided into three parts. The conditions (W1)-(W3) ensure that the functional $\mathcal{W}_t$ satisfies the regularity assumptions (P1)-(P4). The condition (W4) is necessary to perform the discrete to continuous limit in the Euler-Lagrange equations and assumption (W5) will be vital when applying the flow interchange technique to derive improved regularity results for the discrete approximation $\bar{\rho}_{\tau_n,\omega}$. So before using the flow interchange technique, we give two more estimates for the discrete solutions $\left( \rk \right)_\kinN$.

\begin{lemma}[Classical estimates II] \label{MoreBoundsMMS}
Let $\rho_0 \in \mathcal{K}_m$. For fixed $T>0$, there exists a constant $C(T,\tau_*,\rho_0)$, only depending on $T, \tau_*$ and $\rho_0$, such that for all partitions  $\mathcal{T}$, with $\btau$ sufficiently small, the corresponding discrete solutions $\left( \rk \right)_\kinN$ satisfies for all $k$ with $\tk< T$:
		\begin{align}
\M( \rk ) & \leq C(T,\tau_*,\rho_0), \label{BoundSecMom}  \\  
\left|\mathcal{H}(\rk)\right| & \leq C(T,\tau_*,\rho_0). \label{BoundEntropy}
		\end{align}
\end{lemma}


\begin{proof}
Estimate \eqref{BoundSecMom} follows from the fact that one can estimate the second moment in terms of the Wasserstein distance, i.e.,
	\begin{equation*}
			\M(\rk) \leq 2 \W^2(\rk, \mu_* ) + 2 \M(\mu_*).
	\end{equation*}
The first term is bounded by the classical estimate \eqref{bounddn} and hence we established the first bound \eqref{BoundSecMom}. From the Carleman estimate \cite{jko} we can derive that there exists $\gamma \in \left( \frac{d}{d+2} ,1 \right)$ and $C >0 $ such that 
	\begin{equation*}
		- C ( \M(\rk) +1 )^{\gamma} \leq \mathcal{H}(\rk) \leq C ( \mathcal{U}_m (\rk) +1). 
	\end{equation*}
	Hence, estimate \eqref{BoundFk} and the first result yield the desired bound \eqref{BoundEntropy}.
\end{proof}



The next theorem is concerned with an additional regularity result for the piecewise constant interpolation $\overline{\rho}_\tau$. To be more precise, we prove a $H^1$-regularity result for $\overline{\rho}_\tau^{\frac{m}{2}}$ using the flow interchange technique along the heat flow $S^{\mathcal{H}}$ of the entropy functional $\mathcal{H}$.

\begin{thm}[Improved regularity results] \label{h1bounds}
Let $\rho_0 \in \mathcal{K}_m$. For fixed $T>0$, there exists a constant $C(T,\tau_*,\rho_0)$, only depending on $T, \tau_*$ and $\rho_0$, such that for all partitions  $\mathcal{T}$, with sufficiently small $\btau$, the corresponding discrete solution $\left( \rk \right)_\kinN$  satisfies for all $k$ with $\tk< T$:
	\begin{align}
		\left\|\nabla \left(\rk \right)^{\frac{m}{2}}\right\|^2_{\Lzwei} & \leq C(T,\tau_*,\rho_0) \left( 1+ \frac{1}{\tau_{k}} \left( \mathcal{H} ( \rkm ) - \mathcal{H} (\rk) \right) \right), \\
		\left\|  \overline{\rho}^{\frac{m}{2}}_\btau\right\|_{L^2 \left( \tau_1 ,T; H^1\left(\Rd\right) \right)} &  \leq C(T,\tau_*,\rho_0).
	\end{align}
\end{thm}


\begin{proof}

Fix $\rk$ and define $\rho_s$ as the solution of the heat equation with initial datum $\rk$, which corresponds to the $0$-flow $S^{\mathcal{H}}$ of the entropy functional $\mathcal{H}$. By the regularization property of the heat equation one has $\rho_s \in \mathcal{K}_m \cap \mathcal{C}^\infty$ and $\rho_s >0$ for $s>0$. So we can compute the dissipation of the energy along the heat flow and obtain by using (W4)
	\begin{align*}
	\intertext{$m=1$:}
		 - \frac{\mathrm{d}}{\mathrm{d}s} \left(\mathcal{H}+ \tW_\tk \right) (\rho_s) & = - \int_\Rd \log \left( \rho_s(x) \right) \Delta \rho_s (x) \dx - \frac{1}{2} \iint\limits_{\Rd \times \Rd} W_\tk (x,y)	 \left( \Delta_x \rho_s(x) \rho_s(y) + \rho_s(x) \Delta_y \rho_s(y) \right) \dxy \\
		& =  4  \int_\Rd \left\| \nabla \rho_s^{\frac{1}{2}} (x) \right\|^2 \dx -  \iint\limits_{\Rd\times\Rd}  \Delta_x W_\tk (x,y)  \rho_s(x) \rho_s(y)  \dxy \\ 
		& \geq  4 \int_\Rd \left\| \nabla \rho_s^{\frac{1}{2}} (x) \right\|^2 \dx -  d_3( 1 + 2 \M(\rho_s) ) ,
		\intertext{$m>1$:}
		-\frac{\mathrm{d}}{\mathrm{d}s} \left( \mathcal{U}_m + \tW_\tk \right) (\rho_s) & =  - \frac{m}{m-1} \int_\Rd \rho^{m-1}_s(x) \Delta \rho_s (x) \dx - \frac{1}{2} \iint\limits_{\Rd \times \Rd} W_\tk (x,y)	 \left( \Delta_x \rho_s(x) \rho_s(y) + \rho_s(x) \Delta_y \rho_s(y) \right) \dxy \\
		&=  \frac{4}{m} \int_\Rd \left\| \nabla \rho_s^{\frac{m}{2}} (x) \right\|^2 \dx -  \iint\limits_{\Rd\times\Rd}  \Delta_x W_\tk (x,y)  \rho_s(x) \rho_s(y)  \dxy \\ 
		& \geq \frac{4}{m} \int_\Rd \left\| \nabla \rho_s^{\frac{m}{2}} (x) \right\|^2 \dx -  d_3( 1 + 2 \M(\rho_s) ) .
	\end{align*}
	Now, since the the second moment increases at most linearly along the heat flow $\rho_s$, the flow interchange lemma (Theorem \ref{FlowInterchange}) yields with the uniform bounds of the second moments \eqref{BoundSecMom}
		\begin{align}
		\begin{split}
			\liminf_{t \searrow 0} \frac{1}{t} \int_0^t \left\| \nabla \rho_s^{\frac{m}{2}} \right\|_\Lzwei^2 \ ds & \leq \frac{m}{4} \left(\mathcal{D}^{\mathcal{H}}(\mathcal{E}+\mathcal{W}_\tk)(\rk) +  C(T,\tau_*,\rho_0)  \right)  \\
			& \leq  \frac{m}{4} \left( \frac{1}{\tau_{k}} \left( \mathcal{H} ( \rkm ) - \mathcal{H} (\rk) \right) +  C(T,\tau_*,\rho_0)  \right). \label{ApplFlowInterchange}
			\end{split}
		\end{align}
To compute the limes inferior of the l.h.s., define the auxiliary function $D$
	\begin{equation*}
D(t,x) := \frac{1}{t} \int_0^t \rho_s(x)^{\frac{m}{2}} \ \mathrm{d} s  ,
\end{equation*}
which is an element of $\mathcal{C}\left( \left[0,\infty \right); \Lzwei \right)$ and as a consequence $D(t) \rightarrow D(0) = \left( \rk \right)^{\frac{m}{2}}$ in $\Lzwei$ for $t \searrow0$. Due to inequality \eqref{ApplFlowInterchange} and the bounds on the entropy \eqref{BoundEntropy} the gradient of $D (t) $ is bounded in $\Lzwei$ for sufficiently small $t$, since
	\begin{equation*}
		\left\|\nabla D(t, \cdot )\right\|_\Lzwei^2 = \left\| \frac{1}{t} \int_0^t \nabla \rho_s^{\frac{m}{2}} \ \mathrm{d} s \right\|_\Lzwei^2 \leq \frac{1}{ t} \int_0^t \left\| \nabla \rho_s^{\frac{m}{2}} \right\|_\Lzwei^2 \ \mathrm{d} s .
\end{equation*}
Thus  $\nabla D (t,  \cdot ) $ converges weakly in $\Lzwei$ to $\nabla \left( \rk \right)^{\frac{m}{2}}$ and consequently $\nabla \left( \rk \right)^{\frac{m}{2}}\in \Lzwei$. By the weak lower semi-continuity of the $\Lzwei$-norm, we obtain
	\begin{equation*}
		\left\|\nabla \left( \rk \right)^{\frac{m}{2}}\right\|_\Lzwei^2 \leq  \liminf_{t \searrow 0} \frac{1}{t} \int_0^t \left\| \nabla \rho_s^{\frac{m}{2}} \right\|_\Lzwei^2 \ \mathrm{d} s \leq \frac{m}{4} \left(\frac{1}{\tau_{k}} \left( \mathcal{H} ( \rkm ) - \mathcal{H} (\rk) \right) +  C(T,\tau_*,\rho_0)  \right). 
	\end{equation*}
The second claim follows by adding up all those inequalities, using the telescoping sum, and the classical estimate \eqref{BoundFk}, i.e.,
	\begin{align*}
	\left\| \overline{\rho}^{\frac{m}{2}}_\btau\right\|_{L^2 \left( \tau_1,T; H^1\left(\Rd\right) \right)} & \leq  \sum_{k=1}^{k_\btau (T) +1} \tau_k \left(  \left\|\nabla \left(\rk \right)^{\frac{m}{2}}\right\|^2_{\Lzwei} + (m-1) \mathcal{U}_m( \rk) \right)\\
	& \leq C(T,\tau_*,\rho_0) \left( t^{k_\btau (T) +1 }_\btau +   \mathcal{H}(\rho_0) - \mathcal{H} (\rho_\btau^{k_\btau (T)+1} )  \right) \\
	& \leq C(T,\tau_*,\rho_0). \qedhere
	\end{align*}
\end{proof}

In the next theorem we derive approximate Euler-Lagrange equations for the weak formulation, the key idea is the \emph{JKO}-method \cite{jko}, i.e., we determine the first variation in the space $(\Pzwei,\W)$. 


\begin{thm}[Approximate Euler-Lagrange equations] \label{DiscInequaDens}
Given a partition $\mathcal{T}$ with $\btau \in \left(0,\tau_* \right)$ and $\rho_0 \in \mathcal{K}_m$. Let $\left(\rk\right)_\kinN$ be the corresponding discrete solution and define $\bp_\btau^k$ as the optimal transport plan of $\rk$ and $\rkm$. Then we have for all $\psi \in C^\infty_c(\Rd) $
\begin{equation}
 \iint\limits_{\Rd\times \Rd} \nabla_x W_{\tk} (x,y) \cdot \nabla \psi(x) \rk(x) \rk(y) \dxy - \int_{\Rd}  \Delta \psi(x) \rk (x)^m \dx = - \frac{1}{\tau_k} \iint\limits_{\Rd\times\Rd} (y-x) \nabla \psi(y) \ \mathrm{d} \bp_\btau^k(x,y).  \label{MMsequality}
\end{equation}
\end{thm}


\begin{proof}
Fix $\rk, \rkm$ and $\psi$ and introduce the flux function $\Phi_s$, which satisfies
\begin{equation*}
 \frac{\mathrm{d}}{\mathrm{d}s}  \Phi_s (x) = \nabla \psi \left( \Phi_s (x) \right), \qquad \Phi_0(x) = x.
\end{equation*}
The flux function exists, is continuous differentiable with respect to the initial condition and is locally invertible at time $s=0$, since $D \Phi_0 (x) = \mathds{1}$. Consider the perturbation $\rho_s$ of $\rk$ as the push-forward of $\rk$ under $\Phi_s$, i.e., $\rho_s = (\Phi_s)_\# \rk$ and calculate the first variation of the Moreau-Yosida functional, which is given by
	\begin{align*}
		\intertext{$m=1$:}
					\frac{\mathrm{d}}{\mathrm{d}s} \left. \left( \frac{1}{2\tau_k} \W^2(\rho_s,\rkm) + \mathcal{H} (\rho_s)+ \mathcal{W}_\tk (\rho_s)\right)\right|_{s=0}  & \leq \frac{1}{\tau_k} \iint\limits_{\Rd\times\Rd} (y-x) \nabla \psi(y) \ \mathrm{d} \bp_\btau^k(x,y) - \int_\Rd \Delta \psi (x) \rk (x) \dx \\ 
					& \quad +  \iint\limits_{\Rd \times \Rd} \nabla_x W_\tk(x,y) \cdot \nabla \psi(x) \rk(x) \rk(y) \dxy, \\
		\intertext{$m>1$:}
					\frac{\mathrm{d}}{\mathrm{d}s}  \left. \left( \frac{1}{2\tau_k}\W^2(\rho_s,\rkm) +  \mathcal{H}(\rho_s) + \mathcal{W}_\tk (\rho_s) \right) \right|_{s=0} & \leq	\frac{1}{\tau_k} \iint\limits_{\Rd\times\Rd} (y-x) \nabla \psi(y) \ \mathrm{d} \bp_\btau^k(x,y) - \int_\Rd \Delta \psi (x) \rk (x)^m \dx \\
					& \quad  +  \iint\limits_{\Rd \times \Rd}  \nabla_x W_\tk(x,y) \cdot \nabla \psi(x) \rk(x) \rk(y) \dxy.
	\end{align*}
	Since $\rk$ is a minimizer of the Moreau-Yosida functional, the directional derivative has to be greater or equal to zero. Notice that if $\Phi_s$ solves the flux equation for $\nabla \psi$, then $\Phi_{-s}$ solves the flux equation for $-\nabla \psi$. Hence, we obtain the reversed inequality and therefore the desired result.
\end{proof}

This theorem is devoted to the convergence of the interpolation of the densities to a limit curve, which will turn out to be a weak solution.


\begin{thm}[Convergence of piecewise constant interpolation $\overline{\rho}_\btau$] \label{ConvergencePiecewiseConstantInterpolation}
Given a family of partitions $\left(\mathcal{T}_n \right)_\ninN$, with sufficiently small $\btau_n$  and $\btau_n \searrow 0$. For $\rho_0 \in \mathcal{K}_m$ define the corresponding piecewise constant interpolation $\overline{\rho}_{\btau_n}$ and let $\rho_*  \in \mathrm{AC}^2 \left(\left[0, \infty \right), \Paz \right)$ be the limit curve obtained by the Minimizing Movement method from Theorem \ref{ExistenceMMSLimitcurvePC}. Then there exists a non-relabeled subsequence of $\btau_n$ such that for all $\Omega \Subset \Rd$ we have 
		\begin{equation*}
			\overline{\rho}_{\btau_n}  \rightarrow \rho_* \qquad \mathrm{in} \ L^m\left(0,T;L^m(\Omega) \right) \qquad  \mathrm{and \ all} \ T\in \left[0,\infty \right).
		\end{equation*}
\end{thm}


\begin{proof}
Fix $\Omega=\mathds{B}_R(0) \subset \Rd$ and $\tilde{\tau} > 0$. In order to apply the extension of the Aubin-Lions lemma (Theorem \ref{LmLmConv}), define the functional $\mathcal{A}$, via			
		\begin{align*}
		 \mathcal{A}  (\rho) := & \begin{cases} \left\|\rho^{\frac{m}{2}} \right\|^2_{H^1\left(\Omega\right)}& \mbox{if} \ \rho^{\frac{m}{2}} \in H^1\left(\Omega\right), \\
		 + \infty  & \mbox{else}. \end{cases} 
\end{align*}
	This functional $\mathcal{A}$ is clearly measurable, lower semi-continuous and has compact sublevels with respect to the $L^m(\Omega)$ topology. Furthermore, introduce the pseudo-distance $g$ on $L^m(\Omega)$ as follows:
	\begin{align*}
		g\left(\rho,\nu \right) & :=   \inf \left\{ \W( \tilde{\rho},\tilde{\nu}) \left| \tilde{\rho} \in \Sigma(\rho),\ \tilde{\nu} \in \Sigma(\nu) \right. \right\}, \\
		\Sigma(\rho) & := \left\{ \tilde{\rho} \in \Pzwei \left| \tilde{\rho} |_{\Omega}= \rho, \ \M(\tilde{\rho}) \leq C(T,\tau_*,\rho_0) \right. \right\} .
		\end{align*}
		For convenience, we set $g$ equal to infinity if either $\Sigma(\rho) = \emptyset$ or $\Sigma(\nu) = \emptyset$. Since $\Sigma(\rho)$ and $\Sigma(\nu)$ are compact sets with respect to the narrow topology, the infimum is attained by some pair $\tilde{\rho}_*, \tilde{\nu}_*$. Moreover, $g(\rho,\nu)=0$ implies $\tilde{\rho}_*=\tilde{\nu}_*$ in the sense of measures and clearly $\rho=\nu$ a.e. and in $L^m(\Omega)$. To prove the lower semi-continuity of $g$ with respect to the $L^m$-topology choose two convergent sequences $\rho_n\rightarrow \rho, \ \nu_n \rightarrow \nu$ in $L^m(\Omega)$ with $ \sup_n g(\rho_n,\nu_n) < \infty$. Hence, there is a sequence of minimizer $\tilde{\rho}_n,\tilde{\nu}_n$ such that $ g(\rho_n, \nu_n) = \W (\tilde{\rho}_n, \tilde{\nu}_n)$. Since the second moments  are by definition uniformly bounded, we can extract a non-relabeled convergent subsequence which converges narrowly to $\tilde{\rho} \in \Sigma(\rho), \ \tilde{\nu} \in\Sigma(\nu)$. Summarized we obtain
			\begin{equation*}
				g(\rho,\nu) \leq \W ( \tilde{\rho}, \tilde{\nu}) \leq \liminf_\ntoinf \W( \tilde{\rho}_n, \tilde{\nu}_n) = \liminf_\ntoinf g(\rho_n,\nu_n). 
			\end{equation*}
Therefore, $g$ is lower semi-continuous with respect to narrow convergence and is compatible with $\mathcal{A}$, i.e., $g$ satisfies the hypothesis of Theorem \ref{LmLmConv}. Next, we verify the assumptions to the sequence $\left( \left.\overline{\rho}_{\btau_n}\right|_{\Omega} \right)_\ninN$ of Theorem \ref{LmLmConv}. By Theorem \ref{h1bounds} it is clear that $\left( \left.\overline{\rho}_{\btau_n}\right|_{\Omega} \right)_\ninN$ is tight with respect to $\mathcal{A}$, since for sufficiently small $\tau_n$:
		\begin{equation*}
			\sup_n \int_{\tilde{\tau}}^T \mathcal{A}(\left.\overline{\rho}_{\btau_n}\right|_{\Omega}  (t)) \dt = \sup_n \int_{\tilde{\tau}}^T \left\|\overline{\rho}_{\btau_n}^{\frac{m}{2}} (t)    \right\|_{H^1(\Omega)}^2 \dt \leq \sup_n \int_{\tilde{\tau}}^T \left\|\overline{\rho}_{\btau_n}^{\frac{m}{2}}(t)  \right\|_{H^1(\Rd)}^2 \dt \leq C(T, \tau_*, \rho_0).
		\end{equation*}		
		To prove the  relaxed averaged weak integral equi-continuity condition of $\left.\overline{\rho}_{\btau_n}\right|_{\Omega}$ we fix $\btau$ and estimate the inner integrand by
\begin{equation*}
	\int_0^{T-h} g ( \left.\overline{\rho}_{\btau}\right|_{\Omega} (t+h), \left.\overline{\rho}_{\btau}\right|_{\Omega} (t)) \dt  \leq \int_0^{T-h} \W ( \overline{\rho}_{\btau}  (t+h), \overline{\rho}_{\btau} (t)) \dt = \sum_{k=0}^{k_\btau (t)-1} \int_{\tk}^{\tkp} \W (\overline{\rho}_\btau (t+h), \rk) \dt.
\end{equation*}
To calculate this expression define $l_k \in \N$ as the smallest integer such that $ h \leq t_\btau^{k+l_k+1} - \tk $. Then we have
		\begin{align*}
		\sum_{k=0}^{k_\btau (t)-1} 	\int_{\tk}^{\tkp} \W(\overline{\rho}_\btau (t+h), \rk) \dt & \leq \sum_{k=0}^{k_\btau (t)-1}  \tau_k \sum_{j=0}^{l_k } \W(\rho_\btau^{k+j+1}, \rho_\btau^{k+j}) \\
			& \leq \sum_{k=0}^{k_\btau (t)-1}  \tau_k \left( \sum_{j=0}^{l_k } \tau_{k+j+1} \right)^{\frac{1}{2}} \left( \sum_{j=0}^{l_k } \frac{1}{\tau_{k+j+1}}  \W^2(\rho_\btau^{k+j+1}, \rho_\btau^{k+j}) \right)^{\frac{1}{2}} \\
					\intertext{by H\"older's inequality. Consequently, with \eqref{boundsumdk}, one has}
			\sum_{k=0}^{k_\btau (t)-1} 	\int_{\tk}^{\tkp} \W(\overline{\rho}_\btau (t+h), \rk) \dt & \leq  \sum_{k=0}^{k_\btau (t)-1}  \tau_k  \left( \tau_{k+l_k +1} + h \right)^{\frac{1}{2}}  C(T,\tau_*,\rho_0)  \leq  C(T,\tau_*,\rho_0)  T \left(\btau + h \right)^{\frac{1}{2}}.
		\end{align*}
		So, the combination of these results yield
		\begin{equation*}
			\liminf_{h \searrow 0 } \limsup_{\ntoinf} \frac{1}{h} \int_0^h  \int_0^{T-t} g ( \left.\overline{\rho}_{\btau_n}\right|_{\Omega} (s+t), \left.\overline{\rho}_{\btau_n}\right|_{\Omega} (s)) \dss \dt  \leq  \liminf_{h \searrow 0 } \limsup_{\ntoinf} C(T,\tau_*,\rho_0)  \frac{1}{h} \int_0^h \left(\btau_n + t \right)^{\frac{1}{2}} \dt =0.
		\end{equation*}
		Therefore, we can conclude that there exists a non-relabeled subsequence $\overline{\rho}_{\btau_n}$ which converges in $L^m(\mathds{B}_R(0))$ to some limit $\rho_+$, in measure w.r.t. $t\in[0,T]$. Using the uniform bound in $L^\infty(0,T;L^m(\Rd))$, we obtain convergence in $L^m( \tilde{\tau}, T, L^m(\mathds{B}_R(0)))$ by the dominated convergence theorem. Moreover, the limit curves $\rho_*$ and $\rho_+$ have to coincide, since the $\overline{\rho}_\btau$ converges also in measure to $\rho_+$ and to $\rho_*$, so both limits have to be equal. Two diagonal arguments in $\tilde{\tau} \searrow 0$ and $R \rightarrow \infty$ yield the desired convergence result.
\end{proof}


To complete the proof of Theorem \ref{thm:fpe}(a) it remains to verify that $\rho_*$ is indeed a solution to \eqref{FPeq} in the sense of distributions.

\begin{thm}\label{ExistenceFP}
Let $m \in \left[1,\infty\right)$ and $\rho_0\in\Km$. Then there exists a global weak solution $\rho_* \in \mathrm{AC}^2 \left( \left[0,\infty \right]; \Paz \right)$ of the time-dependent Fokker-Planck equation \eqref{FPeq}, i.e., $\rho_*$ satisfies for all $\psi \in \mathcal{C}_c^{\infty} \left( \left[0,\infty \right) \times \Rd \right)$:
			\begin{align*}
			 &  \iiint\limits_{\left[0,T\right]\times\Rd\times \Rd} \nabla_x W_t (x,y) \nabla \psi(t,x) \rho_* (t,x)  \rho_* (t,y)  \mathrm{d} x \mathrm{d} y\mathrm{d} t -\int_0^{T}\int_{\Rd}  \Delta \psi(t,x)  \rho_* (t,x)^m \dx \dt \\
			= & \int_0^T \int_\Rd \rho_*(t,x) \partial_t \psi(t,x) \dx \dt + \int_\Rd \rho_0(x) \psi(0,x)\dx.
			\end{align*}
\end{thm}

\begin{proof}
	Let $\mathcal{T}_n$ be a partition with sufficiently small $\btau_n$ and let $\overline{\rho}_{\btau_n}$ be the corresponding piecewise constant interpolation obtained by the Minimizing Movement scheme. We know by Theorem \ref{ExistenceMMSLimitcurvePC}  this sequence converges locally with respect to the narrow convergence to the limit curve $\rho_* \in \mathrm{AC}^2 \left( \left[0,\infty \right); \Paz \right)$ with $\rho_* (t) \in \mathcal{K}_m$  and by construction we have that $\lim_{t \searrow 0} \rho_*(t) = \rho_*(0) = \lim_{\tau \searrow 0} \rho_{\tau} (0) = \rho_0$ in $\left( \Pzwei , \W \right)$. 

To prove that $\rho_*$ solves the weak formulation fix $\psi \in  \mathcal{C}_c^{\infty} \left( \left[0,\infty \right) \times \Rd \right)$ and choose $T>0$ and $\Omega \Subset \Rd$ such that $supp \ \psi \subset \left[0,T\right]\times\Omega$. Insert $\psi(t^{k-1}_{\btau_n},x)$ for $\psi(x)$ and sum the approximate Euler-Lagrange equations \eqref{MMsequality} from $k=1$ to $k_n(T)+1$ to obtain
	\begin{align}
	\begin{split}
		  &  \int_0^{t_{\btau_n}^{k_n(T)+1}} \iint\limits_{\Rd\times \Rd}  F_n (t,x,y) G_n(t,x) \overline{\rho}_{\btau_n} (t,x)  \overline{\rho}_{\btau_n} (t,y) \dxy \dt - \int_0^{t_{\btau_n}^{k_n(T)+1}}\int_{\Rd}  H_n(t,x) \overline{\rho}_{\btau_n} (t,x)^m \dx \dt \\
			= &  - \sum_{k=1}^{k_n(T)+1}  \ \  \iint\limits_{\Rd\times\Rd} (y-x) \nabla \psi(t^{k-1}_{\btau_n},y) \ \mathrm{d} \bp_{\btau_n}^{k}(x,y)  
			\end{split} \label{SumMMsequality}
	\end{align}
	with the two auxiliary functions $F_n$ and $G_n$, defined by
			\begin{equation*}
		F_n(t,x,y):= \nabla_x W ({t^{k}_{\btau_n}},x,y), \qquad G_n(t,x)= \nabla_x \psi(t^{k-1}_{\btau_n},x ), \qquad  H_n(t,x):= \Delta \psi(t^{k-1}_{\btau_n},x ) \qquad \mathrm{for} \ t \in \left(t^{k-1}_{\btau_n}, t^{k}_{\btau_n} \right].
	\end{equation*}	
	Now consider each integral in \eqref{SumMMsequality} separately. To calculate the limit of the first integral we begin with proving 
			\begin{equation}
			\lim_{\ntoinf} \int_\Rd F_n(t,x,y) \overline{\rho}_{\btau_n}(t,y) \dy = \int_\Rd \nabla_x W(t,x,y) \rho_*(t,y) \dy \label{AuxConvExistLimit}
		\end{equation}	
	and that this limit is uniform with respect to $t$ and $x$. So fix $t \in \left[0,T\right]$ and $x \in \Omega$ and let $d_4,r,\tilde{\alpha}$ as in assumption (W5). Further, fix some arbitrary $R>0$ and split up the difference of the two terms, i.e.,
		\begin{align*}
			\left\|  \int_\Rd F_n(t,x,y) \overline{\rho}_{\btau_n}(t,y) - \nabla_x W(t,x,y) \rho_*(t,y) \dy  \right\|\leq&   \left\| \int_{\mathds{B}_R} \nabla_x W(t,x,y)  \left( \overline{\rho}_{\btau_n}(t,y) - \rho_*(t,y)  \right) \dy   \right\| \\
			 + & \left\| \int_{\mathds{B}^c_R} \nabla_x W(t,x,y)  \left( \overline{\rho}_{\btau_n}(t,y) - \rho_*(t,y)  \right) \dy   \right\| \\
			 +& \left\|\int_\Rd \overline{\rho}_{\btau_n}(t,y) \left( F_n(t,x,y) - \nabla_x W(t,x,y) \right)  \dy  \right\| .
		\end{align*}
The first term converges to zero, as $\ntoinf$, since $\nabla_x W$ is continuous, the domain of integration is compact, and $\overline{\rho}_{\btau_n}(t)$ converges narrowly to $\rho_*(t)$. To bound the second term we use the assumption (W5) to obtain
		\begin{align*}
			\left\| \int_{\mathds{B}^c_R} \nabla_x W(t,x,y)  \left( \overline{\rho}_{\btau_n}(t,y) - \rho_*(t,y)  \right) \dy   \right\| & \leq d_4 \int_{\mathds{B}^c_R} \left(1+ \left\|y\right\|^r \right) \left( \overline{\rho}_{\btau_n} (t,y) + \rho_*(t,y)  \right) \dy  \\
			& \leq d_4 \left( \int_{\mathds{B}^c_R} \left(\overline{\rho}_{\btau_n} (t,y) + \rho_*(t,y) \right) \dy + R^{r-2} \left( \M(  \overline{\rho}_{\btau_n} (t) )+ \M( \rho_*(t) ) \right) \right). 
		\end{align*}
	The first integral converges, as $\ntoinf$, due to the Portmanteau Theorem \cite{klenke} and the r.h.s. is bounded by the classical estimate \eqref{BoundSecMom}. Hence we obtain in the limit $\ntoinf$
			\begin{equation*}
				\limsup_\ntoinf \left\| \int_{\mathds{B}^c_R} \nabla_x W(t,x,y)  \left( \overline{\rho}_{\btau_n}(t,y) - \rho_*(t,y)  \right) \dy   \right\|  \leq d_4 \left( 2 \int_{\mathds{B}^c_R} \rho_*(t,y) \dy +  R^{r-2} C(T,\tau_*, \rho_0) \right).
				\end{equation*}
	Since $R$ was arbitrarily chosen and $r<2$, we conclude, with $R \rightarrow \infty$, the limit of this term is equal to zero. The third term converges also to zero, since by the absolute continuity assumption of (W5) we have
			\begin{align*}
				\left\|\int_\Rd \overline{\rho}_{\btau_n}(t,y) \left( F_n(t,x,y) - \nabla_x W(t,x,y) \right)  \dy  \right\|  & \leq \int_\R \int_{t_{\btau_n}^{k_n(t)}}^t \tilde{\alpha}(r) \dr \left(1+ \left\|y\right\|^2 \right) \overline{\rho}_{\btau_n}(t,y)  \dy \\
				&\leq \int_{t- \btau_n }^t \tilde{\alpha}(r) \dr \left(1 + \M(\overline{\rho}_{\btau_n}(t,y) ) \right).
			\end{align*}
			The first factor converges uniformly to zero and the second factor is bounded by the classical estimate \eqref{BoundSecMom}, again. Hence, we have shown that the limit in \eqref{AuxConvExistLimit} is uniform in $x$ and $t$. This result in combination with the uniform convergence of $G_n$ to $\nabla_x \psi(t,x)$ and the $L^m(0,T;L^m(\Omega))$-convergence of $\overline{\rho}_{\btau_n}$ to $\rho_*$, we deduce 
	\begin{align*}
\lim_\ntoinf \int_0^{t_{\btau_n}^{k_n(T)+1}} \iint\limits_{\Rd\times \Rd}  \tilde{W}_n (t,x,y) \overline{\rho}_{\btau_n} (t,x)  \overline{\rho}_{\btau_n} (t,y) \dxy \dt =   \iiint\limits_{\left[0,T\right]\times\Rd\times \Rd} \nabla_x W_t (x,y) \nabla \psi(t,x) \rho_* (t,x)  \rho_* (t,y)  \mathrm{d} x \mathrm{d} y\mathrm{d} t.
	\end{align*}
The limit of the second integral of \eqref{SumMMsequality} follows by the uniform convergence of $H_n$ to $\Delta \psi$ and by the $L^m(0,T;L^m(\Omega))$-convergence of $\overline{\rho}_{\btau_n}$, i.e.:
	\begin{equation*}
	\lim_\ntoinf \int_0^{t_{\btau_n}^{k_n(T)+1}}\int_{\Rd}  \tilde{\psi}_n(t,x) \overline{\rho}_{\btau_n} (t,x)^m \dx \dt = \int_0^{T}\int_{\Rd}  \Delta \psi(t,x)  \rho_* (t,x)^m \dx \dt.
	\end{equation*}
In order to calculate the limit of the right-hand-side of \eqref{SumMMsequality} notice that by Taylor's formula we can expand the integrand as follows
	\begin{align*}
		 & \sum_{k=1}^{k_n(T)+1}   \iint\limits_{\Rd\times\Rd} (y-x) \nabla \psi(t^{k-1}_{\btau_n},y) \ \mathrm{d} \bp_{\btau_n}^{k}(x,y)   \\
		 =& \sum_{k=1}^{k_n(T)+1}   \iint\limits_{\Rd\times\Rd}  \psi(t^{k-1}_{\btau_n},x) - \psi(t^{k-1}_{\btau_n},y) + \mathcal{O}( \left\|x-y\right\|^2 ) \ \mathrm{d} \bp_{\btau_n}^{k}(x,y) \\
		=& \sum_{k=1}^{k_n(T)+1}    \int_\Rd \left( \rho^{k}_{\btau_n}(x) - \rho^{k-1}_{\btau_n}(x) \right) \psi(t^{k-1}_{\btau_n},x) \dx + \sum_{k=1}^{k_n(T)+1}  \mathcal{O}(\W^2(  \rho^{k}_{\btau_n},\rho^{k-1}_{\btau_n}) ).
		\end{align*}
		Furthermore, rearrange the first term and use the classical estimate \eqref{boundsumdk} to bound the second term, to obtain
		\begin{align*}
		\sum_{k=1}^{k_n(T)+1}   \iint\limits_{\Rd\times\Rd} (y-x) \nabla \psi(t^{k-1}_{\btau_n},y) \ \mathrm{d} \bp_{\btau_n}^{k}(x,y)=&  - \int_0^T \int_\Rd \overline{\rho}_{\btau_n}(t,x) \partial_t \psi(t,x) \dx \dt - \int_\Rd \rho_0(x) \psi(0,x)\dx + \mathcal{O}( \btau_n ).
	\end{align*}
In combination with the narrow convergence of $\overline{\rho}_{\btau_n}$ we can calculate the limit of the third term of \eqref{SumMMsequality} and is given by
		\begin{equation*}
			\lim_\ntoinf \sum_{k=1}^{k_n(T)+1}   \iint\limits_{\Rd\times\Rd} (y-x) \nabla \psi(t^{k-1}_{\btau_n},y) \ \mathrm{d} \bp_{\btau_n}^{k}(x,y) \dxy = - \int_0^T \int_\Rd \rho_*(t,x) \partial_t \psi(t,x) \dx \dt - \int_\Rd \rho_0(x) \psi(0,x)\dx. 
		\end{equation*}
		These results shows that $\rho_*$ satisfies
			\begin{align*}
			 &  \iiint\limits_{\left[0,T\right]\times\Rd\times \Rd} \nabla_x W_t (x,y) \nabla \psi(t,x) \rho_* (t,x)  \rho_* (t,y)  \mathrm{d} x \mathrm{d} y\mathrm{d} t -\int_0^{T}\int_{\Rd}  \Delta \psi(t,x)  \rho_* (t,x)^m \dx \dt \\
			= & \int_0^T \int_\Rd \rho_*(t,x) \partial_t \psi(t,x) \dx \dt + \int_\Rd \rho_0(x) \psi(0,x)\dx,
			\end{align*}
			yielding that $\rho_*$ is a distributional solution to \eqref{FPeq}.
\end{proof}


\section{High-frequency limit}\label{sec:high}

This section is devoted to the time-homogenization of the evolution systems and especially to the high-frequency limit of the Fokker Planck equation, both in the $\lambda$-convex and not-$\lambda$-convex case. As shown in the previous chapter this scheme is well-defined (Theorem \ref{ExistenceMMS}) and $\overline{u}_{\btau,\omega}$ converges to a curve of steepest descent $u_\omega$ of the functional $\mathcal{E} + \mathcal{P}_{\omega t}$ (Theorem \ref{thm:abstractgf}(a)). In order to show that the limit, as $\omega \rightarrow \infty$, of the solutions $u_\omega$ exists, in either the case considered in chapter \ref{FPeq} or in the $\lambda$-convex case considered in \cite{fg}, we have to refine the assumptions to the perturbation term $\mathcal{P}_t$ and prove an $\omega$-independent version of the classical estimates from Theorem \ref{BoundsMMS}.
	
		\begin{assumption} \label{DefRegPHF}
		The periodic perturbation functional $\mathcal{P}_t: \left[0,\infty \right) \times \X \rightarrow \R$ satisfies the regularity conditions (P1)--(P4) and additionally:
			\begin{enumerate}[({P}1)]
\setcounter{enumi}{4}
				\item There exists $L \geq 0$ such that for all $u,v \in \X$ and $t \in \left[0,\infty \right)$
					\begin{equation*}
						\left|(\mathcal{P}_t(u) - \overline{\mathcal{P}}(u)) - (\mathcal{P}_t(v)-\overline{\mathcal{P}}(v)) \right| \leq L  \bd(u,v),
					\end{equation*}
					where $\overline{\mathcal{P}}$ denotes the mean of $\mathcal{P}_t$.
			\end{enumerate}	
	\end{assumption}

This spatial Lipschitz-continuity of the perturbation term $\mathcal{P}_t$ yields now the classical bounds \emph{independent} of  $\omega$.

\begin{lemma}[Classical estimates revisited]\label{BoundsMMsHF}
Let $u_0 \in \mathcal{D}\left(\mathcal{E}\right)$ and $u_\omega$ the solution of the Fokker-Planck equation \eqref{FPeq} obtained by the Minimizing movement scheme. For fixed $T>0$, there exists a constant $C(T,\tau_*,u_0)$, only depending on $T, \tau_*$ and $u_0$, such that for all $\omega >0$ and $t\leq T$ there holds:
		\begin{align}
		\left\||u_\omega '|\right\|_{L^2(0,T)} & \leq  C(T,\tau_*,u_0) , \label{boundsumdkw} \\
		\mathcal{E}(u_\omega(t)) & \leq C(T,\tau_*,u_0), \label{BoundFkw} \\
		\bd^2(u_*, u_\omega (t) )  & \leq C(T,\tau_*,u_0) \label{bounddnw}.
		\end{align}
\end{lemma}


\begin{proof}
We will prove the estimates on a discrete level and then we use the lower semi-$\bs$-continuity of each bound to obtain the desired result. For this purpose fix a partition $\mathcal{T}$, with $\btau < \frac{\tau_*}{2}$, and let $\left( \ukw \right)_\kinN$ be the corresponding discrete solution. As in the proof of Theorem \ref{BoundsMMS} we derive for the discrete solutions $(\ukw)_\kinN$ the inequality \eqref{AuxClasEstMMs}. Exploit the Lipschitz-continuity of $\mathcal{P}_t$ and use Young's inequality to further estimate the right-hand-side to obtain
\begin{align*}
\sum_{k=1}^N \frac{1}{2\tau_k } \dkw & \leq \mathcal{E}(u_0) - \eknw + \sum_{k=1}^N \left( \pkmw - \pkw \right)  \\
	& \leq \mathcal{E}(u_0) - \eknw + \sum_{k=1}^N L \bd(\ukw , \ukmw ) + \overline{\mathcal{P}}(\ukmw) - \overline{\mathcal{P}}(\ukw) \\
		& \leq \mathcal{E}(u_0) + \overline{\mathcal{P}}(u_0) - \eknw  - \overline{\mathcal{P}}(\uknw ) + \sum_{k=1}^N \left(L^2 \tau_k + \frac{1}{4\tau_k} \bd^2(\ukw, \ukmw) \right).
			\end{align*}
			A kick-back argument and the coercivity of $\mathcal{E}$ and $\overline{\mathcal{P}}$ yields now
				\begin{equation*}
				\sum_{k=1}^N \frac{1}{4\tau_k } \dkw  \leq \mathcal{E}(u_0) +  \overline{\mathcal{P}}(u_0)  +  \frac{1}{2\tau_*}\bd^2(u_*,\uknw) - c_* + T L^2.
				\end{equation*}
	Perform a to the foregoing proof similar calculation to get
		\begin{equation*}
			\bd^2(u_*,\uknw) \leq 2 \tau_* \left( \mathcal{E}(u_0) +  \overline{\mathcal{P}}(u_0)  -c_* +T L^2 \right) + 2 \bd^2(u_*,u_0) + \frac{2}{\tau_*} \sum_{k=1}^N \tau_k \bd^2(u_*, \ukw).
		\end{equation*}
	Notice that every constant appearing in this equation is independent of $\omega$. Now repeat the remaining part of the proof with the discrete Gronwall's lemma \cite[Lemma 3.2.4]{ags} to get the desired estimates on the discrete level, which are independent of $\omega$, i.e., we have for all $N$ with $t^N_\btau < T$:
		\begin{align}
		\sum_{k=1}^N \frac{1}{2\tau_k} \dkw & \leq  C(T,\tau_*,u_0) , \qquad \mathcal{E}(\uknw) \leq C(T,\tau_*,u_0), \ \qquad	\bd^2(u_*,\uknw)  \leq C(T,\tau_*,u_0) \label{boundw}.
		\end{align}
	 Now consider a family $\mathcal{T}_n$ of admissible partitions, then the corresponding piecewise constant interpolation $\overline{u}_{\btau_n,\omega}$ converges with respect to the weak topology $\bs$ to $u_\omega$. Since $\mathcal{E}$ and $\bd$ are lower semi-$\bs$-continuous, we have
		\begin{align*}
		\left\||u_\omega '|\right\|_{L^2(0,T)} & \leq \liminf_\ntoinf \sum_{k=1}^{k_n(T)+1} \frac{1}{2\tau_k} \dkw  \leq  C(T,\tau_*,u_0) , \\
			\mathcal{E}(u_\omega(t)) & \leq \liminf_{\ntoinf} \mathcal{E}( \overline{u}_{\btau_n, \omega}(t)) \leq C(T,\tau_*,u_0), \\
			\bd^2(u_*, u_\omega (t) ) &\leq \liminf_{\ntoinf}  \bd^2(u_*, \overline{u}_{\btau_n, \omega}(t)) \leq C(T,\tau_*,u_0).	\qedhere		
		\end{align*}
	\end{proof}
	
		\subsection{High-frequency limit for abstract $\lambda$-convex gradient flows}

Next, we show the second part of Theorem \ref{thm:abstractgf}, i.e., we prove by a direct calculation that in the $\lambda$-convex case the sequence of solutions $u_\omega$ will converge to the solution $u_\infty$ of the time-averaged evolution system. In particular, $\lambda$-convexity of a time-dependent functional $\mathcal{F}_t$ in a metric space is here defined as follows, e.g. see \cite{fg}: There exists a function $\lambda(t)$ such that for every triple $u,v_0,v_1 \in \X$, there exists a curve $\gamma:\left[0,1\right] \rightarrow \X$ with $\gamma(0)=v_0, \ \gamma(1)=v_1$ and 
						\begin{equation}
						\Phi(\tau,t,u,\gamma_s) \leq (1-s) \Phi(\tau,t,u, v_0) + s \Phi(\tau,t,u, v_1) - \frac{1}{2} \left( \frac{1}{\tau} + \lambda(t) \right) s(1-s) \bd^2(v_0,v_1) . \label{Stronglambdaconvex}
					\end{equation}
where $\Phi$ is the Moreau-Yosida functional of $\mathcal{F}_t$. It is known, see \cite{ags} for further details, that $\mathcal{H}$ and $\mathcal{U}_m$ are $0$-convex and that $\mathcal{W}_t$ is $\lambda(t)$-convex if and only if $W_t$ is $\lambda(t)$-convex in $\Rd\times \Rd$. For a derivation of an exact convergence rate, we will use the evolution variational inequality given in \cite[Equation (5.15)]{fg}. Note that the classical estimates do not rely on the compactness of the sublevels of $\mathcal{E}$ or the abstract $\lambda$-convexity of the Moreau-Yosida functional.

		\begin{proof}
		Without loss of generality we assume that the mean of $\mathcal{P}_t$ is equal to zero (otherwise redefine $\tilde{\mathcal{E}}:=\mathcal{E} + \overline{\mathcal{P}}$ and $\tilde{\mathcal{P}_t}:= \mathcal{P}_t - \overline{\mathcal{P}}$), and with $\lambda \leq 0$ be constant. By \cite[Equation (5.15)]{fg} and \cite[Thm. 4.0.4]{ags}, $u_\omega$ and $u_\infty$ satisfy the following two evolution variational inequalities, respectively for all $v\in \X$: 
	\begin{align}
		\frac{1}{2} \bd^2(u_\omega(t),v) - \frac{1}{2}\bd^2(u_\omega (s),v) + \int_s^t \frac{\lambda}{2} \bd^2(u_\omega (r),v) + \mathcal{E}(u_\omega (r))+ \mathcal{P}_{\omega r}(u_\omega (r))\dr & \leq \int_s^t \mathcal{E}(v) + \mathcal{P}_{\omega r}(v)\dr, \label{EVIw} \\
		\frac{1}{2} \bd^2(u_\infty(t),v) - \frac{1}{2}\bd^2(u_\infty (s),v) + \int_s^t \frac{\lambda}{2} \bd^2(u_\infty (r),v) + \mathcal{E}(u_\infty (r)) \dr & \leq \int_s^t \mathcal{E}(v)\dr. \label{EVIa}
	\end{align}
		In order to prove the statement, we apply a Gronwall type argument, i.e., we differentiate the square of the distance of $u_\omega(t)$ and $u_\infty(t)$. Since the solution curves are absolutely continuous, this step is valid and we can apply \cite[Lemma 4.3.4]{ags} to obtain
			\begin{align*}
			\left.\frac{ \mathrm{d}}{\mathrm{d} s } \frac{1}{2} \bd^2(u_\omega(s) , u_\infty(s) ) \right|_{s=t}  \leq & \limsup_{h \searrow 0} \frac{ \frac{1}{2}  \bd^2(u_\omega(t) , u_\infty(t) ) - \frac{1}{2}  \bd^2(u_\omega(t-h) , u_\infty(t) ) }{h}  \\ + & \limsup_{h\searrow 0} \frac{ \frac{1}{2}  \bd^2(u_\omega(t) , u_\infty(t+h) ) - \frac{1}{2}  \bd^2(u_\omega(t) , u_\infty(t) ) }{h}.
			\end{align*}
			The first term on the right-hand-side can be estimated using the evolution variational equation \eqref{EVIw}, the lower semi-$\bs$-continuity of $\mathcal{E}$ and $\bd$, and the $\bs$-continuity of $\mathcal{P}_t$ to get
				\begin{align*}
					& \limsup_{h\searrow 0} \frac{ \frac{1}{2}   \bd^2(u_\omega(t) , u_\infty(t) ) -  \frac{1}{2} \bd^2(u_\omega(t-h) , u_\infty(t) ) }{h} \\
					\leq&  \limsup_{h\searrow 0 }\frac{1}{h} \int_{t-h}^t \left[ \mathcal{E}(u_\infty(t) ) + \mathcal{P}_{\omega r}(u_\infty(t) ) - \frac{\lambda}{2} \bd^2(u_\omega (r), u_\infty(t) ) - \mathcal{E}(u_\omega (r))- \mathcal{P}_{\omega r}(u_\omega (r)) \right]\dr  \\
					\leq &  \mathcal{E}(u_\infty(t) ) + \mathcal{P}_{\omega t}(u_\infty(t) ) - \frac{\lambda}{2} \bd^2(u_\omega (t), u_\infty(t) ) - \mathcal{E}(u_\omega (t))- \mathcal{P}_{\omega t}(u_\omega (t)) .
				\end{align*}
		Analogously, using the evolution variational equation \eqref{EVIa} we obtain for the limit of the second term
			\begin{equation*}
				\limsup_{h\searrow 0} \frac{ \frac{1}{2} \bd^2(u_\omega(t) , u_\infty(t+h) ) - \frac{1}{2}  \bd^2(u_\omega(t) , u_\infty(t) ) }{h} \leq  \mathcal{E}(u_\omega (t))-  \frac{\lambda}{2} \bd^2(u_\omega (t), u_\infty(t) ) - \mathcal{E}(u_\infty(t) )  .
			\end{equation*}
			Hence, by adding these results we get the following estimate 
				\begin{align*}
					\frac{ \mathrm{d}}{\mathrm{d} s } \left. \frac{1}{2}  \bd^2(u_\omega(s) , u_\infty(s) ) \right|_{s=t} \leq   \mathcal{P}_{\omega t}(u_\infty(t) ) -\mathcal{P}_{\omega t}(u_\omega (t)) - \lambda\bd^2(u_\omega (t), u_\infty(t) )
				\end{align*}
				from which we conclude with the differential form of Gronwall's inequality that 
					\begin{align*}
						\exp\left[  2 \lambda t  \right]  \bd^2(u_\omega(t) , u_\infty(t) )  \leq&  \int_0^t \exp\left[2  \lambda r \right] \left(  \mathcal{P}_{\omega r}(u_\infty(r) ) -\mathcal{P}_{\omega r}(u_\omega (r)) \right) \dr.
						\end{align*}
						After a rescaling of the time variable the r.h.s. can be decomposed, i.e.,
						\begin{align*}
						 \int_0^t \exp\left[ 2 \lambda r \right] \left(  \mathcal{P}_{\omega r}(u_\infty(r) ) -\mathcal{P}_{\omega r}(u_\omega (r)) \right) \dr = & \frac{1}{\omega} \int_0^{\omega t} \exp \left[2 \lambda \frac{r}{\omega} \right]   \left(  \mathcal{P}_{r}(u_\infty(\frac{r}{\omega}) ) -\mathcal{P}_{ r}(u_\omega (\frac{r}{\omega})) \right) \dr \\
						 = & \sum_{k=0}^{ \left\lfloor \omega t \right\rfloor-1} \frac{1}{\omega} \int_k^{k+1} \ \exp \left[2 \lambda \frac{r}{\omega} \right]   \left(  \mathcal{P}_{r}(u_\infty(\frac{r}{\omega}) ) -\mathcal{P}_{ r}(u_\omega (\frac{r}{\omega})) \right) \dr \\
						&  \ \ \ \ \ + \frac{1}{\omega} \int_{ \left\lfloor \omega t \right\rfloor}^{\omega t} \ \  \exp \left[2 \lambda \frac{r}{\omega} \right]   \left(  \mathcal{P}_{r}(u_\infty(\frac{r}{\omega}) ) -\mathcal{P}_{ r}(u_\omega (\frac{r}{\omega})) \right) \dr.
					\end{align*}
					The first term in this equation can be further expanded inserting two productive zeros, such that we have
					\begin{align*}
					\frac{1}{\omega} \sum_{k=0}^{ \left\lfloor \omega t \right\rfloor -1} \int_k^{k+1}  &  \ \exp \left[ 2 \lambda \frac{r}{\omega} \right]  \left(  \mathcal{P}_{r}(u_\infty(\frac{r}{\omega}) ) -\mathcal{P}_{ r}(u_\omega (\frac{r}{\omega})) \right) \dr \\
						= \frac{1}{\omega} \sum_{k=0}^{ \left\lfloor \omega t \right\rfloor-1} \int_k^{k+1} & \left[ \exp \left[ 2 \lambda \frac{r}{\omega} \right]   \left(  \mathcal{P}_{r}(u_\infty(\frac{r}{\omega}) ) -\mathcal{P}_{ r}(u_\omega (\frac{r}{\omega})) \right)  -  \exp \left[ 2 \lambda \frac{k}{\omega} \right]   \left(  \mathcal{P}_{r}(u_\infty(\frac{k}{\omega}) ) -\mathcal{P}_{ r}(u_\omega (\frac{k}{\omega})) \right)  \right.  \\
					 +& \left.  \left[\exp \left[ 2 \lambda \frac{r}{\omega} \right]- \exp \left[2  \lambda \frac{r}{\omega} \right] \right]   \left(  \mathcal{P}_{r}(u_\infty(\frac{k}{\omega}) ) -\mathcal{P}_{ r}(u_\omega (\frac{k}{\omega})) \right) \right] \dr \\ 
					= \frac{1}{\omega} \sum_{k=0}^{ \left\lfloor \omega t \right\rfloor-1} \int_k^{k+1} &  		\left[ \exp \left[2  \lambda \frac{r}{\omega} \right]   \left(  \mathcal{P}_{r}(u_\infty(\frac{r}{\omega}) ) -\mathcal{P}_{ r}(u_\omega (\frac{r}{\omega}) ) - \mathcal{P}_{r}(u_\infty(\frac{k}{\omega}) ) + \mathcal{P}_{ r}(u_\omega (\frac{k}{\omega})) \right)  \right. \\	 
								  +      &   \left. \ \exp \left[ 2 \lambda \frac{\zeta(r)}{\omega} \right]   \frac{2 \lambda (r-k)}{\omega}  \left(  \mathcal{P}_{r}(u_\infty(\frac{k}{\omega}) ) -\mathcal{P}_{ r}(u_\omega (\frac{k}{\omega})) \right)  \right] \dr. 	
					\end{align*}
					 for some $\zeta(r) \in \left[k, r \right]$. Exploit the Lipschitz continuity of $\mathcal{P}_t$, the absolute continuity of $u_\omega$, respectively of $u_\infty$, and use the estimates \eqref{boundsumdkw} and \eqref{bounddnw} to obtain the following upper bound for the modulus of the previous equation
						\begin{align*}
										& \left|	   \frac{1}{\omega} \sum_{k=0}^{ \left\lfloor \omega t \right\rfloor -1} \int_k^{k+1}  \exp \left[2  \lambda \frac{r}{\omega} \right]  \left(  \mathcal{P}_{r}(u_\infty(\frac{r}{\omega}) ) -\mathcal{P}_{ r}(u_\omega (\frac{r}{\omega})) \right) \dr   \right| \\
						\leq & \frac{1}{\omega} \sum_{k=0}^{ \left\lfloor \omega t \right\rfloor-1} \int_k^{k+1} 	\left[ \exp \left[ 2 \lambda \frac{r}{\omega} \right]   \left(  L \bd(u_\infty(\frac{r}{\omega}), u_\infty(\frac{k}{\omega}) ) + L \bd(u_\omega(\frac{r}{\omega}), u_\omega (\frac{k}{\omega}) )   \right) \right. \\
							& \qquad \qquad \quad \ \ \left. +    \exp \left[ 2 \lambda \frac{\zeta(r)}{\omega} \right] \frac{2 \left|\lambda\right|(r-k)}{\omega}   L \bd(u_\infty(\frac{k}{\omega}), u_\omega(\frac{k}{\omega}) ) \right]\dr \\
							\leq & \frac{L}{\omega} \sum_{k=0}^{ \left\lfloor \omega t \right\rfloor-1} \int_k^{k+1} 	\left[    \int_{\frac{k}{\omega}}^{\frac{r}{\omega} } |u'_\infty|(s) \dss + \int_{\frac{k}{\omega}}^{\frac{r}{\omega} } |u'_\omega|(s) \dss +\frac{2 \left|\lambda\right|}{\omega} \bd(u_\infty(\frac{k}{\omega}), u_\omega(\frac{k}{\omega}) )   \right] \dr \\
							\leq &\frac{L}{\omega} \int_0^{ \frac{\left\lfloor \omega t \right\rfloor}{\omega}} |u'_\infty|(s) + |u'_\omega|(s)  \dss +  \frac{L}{\omega} \sum_{k=0}^{ \left\lfloor \omega t \right\rfloor-1}  \frac{ 2 \left|\lambda\right|}{\omega} \bd(u_\infty(\frac{k}{\omega}), u_\omega(\frac{k}{\omega}) )   \\
							\leq & \frac{L\sqrt{T} }{\omega}   \left\| |u'_\infty| \right\|_{L^2(0,T)}+ \frac{ L\sqrt{T} }{\omega} \left\| |u'_\omega|\right\|_{L^2(0,T)}  + \frac{ 2 L \left|\lambda \right| T }{\omega}   C(T,\tau_*,u_0) \\
							\leq &  C(T,\tau_*,u_0) \frac{1}{\omega}.
					\end{align*}
					We estimate the remainder term of the starting equation accordingly with a combination of the Lipschitz continuity of $\mathcal{P}_t$ and estimate \eqref{bounddnw} such that we obtain
					\begin{align*}
\left|	 \frac{1}{\omega} \int_{ \left\lfloor \omega t \right\rfloor}^{\omega t} \exp \left[2 \lambda \frac{r}{\omega} \right]   \left(  \mathcal{P}_{r}(u_\infty(\frac{r}{\omega}) ) -\mathcal{P}_{ r}(u_\omega (\frac{r}{\omega})) \right) \dr \right|& \leq \frac{ \omega t -\left\lfloor \omega t \right\rfloor }{\omega} C(T,\tau_*,u_0) \leq C(T,\tau_*,u_0) \frac{1}{\omega}.
\end{align*}
Thus, combining these results we obtain
	\begin{equation*}
		\exp\left[ 2 \lambda t  \right] \frac{1}{2} \bd^2(u_\omega(t) , u_\infty(t) ) \leq C(T,\tau_*,u_0) \frac{1}{\omega}
	\end{equation*}
	yielding the desired uniform convergence rate of the $\rho_{\omega}$ to $\rho_\infty$ for every finite horizon $T$. 
		\end{proof}

		\subsection{High-frequency limit for the nonlinear Fokker-Planck equation}

We show that the second part of the main Theorem \ref{thm:fpe}, i.e., that also in the non-$\lambda$-convex case the solutions $\rho_\omega$ of the Fokker-Planck equation converge and the limit curve $\rho_\infty$ solves the time averaged Cauchy problem \eqref{limFPeq}--\eqref{eq:limicond}. For this purpose we  will make use of assumption (W6), which guarantees that also (P5) is satisfied and therefore also the $\omega$-independent classical estimates Lemma \ref{BoundsMMsHF} are valid.

\begin{proof}
At first, we prove the existence of a narrow convergent subsequence, which converges to an absolutely continuous curve. Note, by the estimates \eqref{BoundFkw} and \eqref{bounddnw} the densities $\rho_{\omega_n} (t)$ are contained for all $t \leq T$ in a set $K$, which is compact with respect to the narrow convergence. Moreover, by estimate \eqref{boundsumdkw} there exists a non-relabeled subsequence such that the metric velocity $|\rho_{\omega_n} '|$ converges weakly in $L^2(0,T)$ to $A \in L^2(0,T)$. To apply the Arzelà-Ascoli theorem we estimate now
	\begin{equation*}
		\limsup_{\omega \rightarrow \infty} \W (\rho_{\omega_n} (t), \rho_{\omega_n} (s)) \leq\limsup_{n \rightarrow \infty}\int_s^t |\rho_{\omega_n}  '|(r) \dr = \int_s^t A(r) \dr.
	\end{equation*}
Hence, by the refined Arzelà-Ascoli theorem \cite[Proposition 3.3.1]{ags} we obtain the existence of a limit curve $\rho_\infty \in \mathrm{AC}^2\left( \left[0,\infty\right),\Paz \right)$ such that $\rho_{\omega_n}$ converges pointwise with respect to the narrow convergence.

To obtain the $L^m(0,T;L^m(\Omega))$-convergence result, fix some $\Omega \Subset \Rd$ and note that one has also the improved regularity result with a constant $C(T,\tau_*, \rho_0)$ independent of $\omega$, i.e.,
			\begin{equation}
				\left\|  \rho_{\omega_n}^{\frac{m}{2}} \right\|_{L^2 \left( 0 ,T; H^1\left(\Rd\right) \right)}   \leq C(T,\tau_*,\rho_0), \label{HFLimitAuxIneq}
		\end{equation}
	since by Theorem \ref{BoundsMMsHF} the constants, appearing in the proof of Theorem \ref{h1bounds}, do not depended on $\omega$. Again, we apply the extension of the Aubin-Lions lemma (Theorem \ref{LmLmConv}) to the sequence $\rho_{\omega_n}$ with the auxiliary functionals $\mathcal{A}$ and $g$ as in the proof of Theorem \ref{ConvergencePiecewiseConstantInterpolation}. As before, due to the estimate \eqref{HFLimitAuxIneq}, the family $\rho_{\omega_n}$ is tight with respect to the normal, coercive integrand $\mathcal{A}$. To verify the relaxed averaged weak integral equi-continuity with respect to $g$ use that $|\rho_{\omega_n}'|$ converges weakly in $L^2_{\mathrm{loc}}\left( \left[0,\infty \right) \right)$ to $A$. This yields
		\begin{align*}
		\liminf_{h \searrow 0 } \limsup_{\ntoinf} \frac{1}{h} \int_0^h  \int_0^{T-t} g ( \left.\rho_{\omega_n}\right|_{\Omega} (s+t), \left.\rho_{\omega_n}\right|_{\Omega} (s)) \mathrm{d} s \dt &  \leq \liminf_{h \searrow 0 } \limsup_{\ntoinf} \frac{1}{h} \int_0^h  \int_0^{T-t} \W ( \rho_{\omega_n}(s+t), \rho_{\omega_n} (s)) \dss\dt \\
		& \leq \liminf_{h \searrow 0 } \limsup_{\ntoinf} \frac{1}{h} \int_0^h  \int_0^{T-t} \int_s^{s+t} |\rho_{\omega_n}'|(r)\dr \dss \dt \\	
		& \leq \liminf_{h \searrow 0 } \frac{1}{h} \int_0^h  \int_0^{T-t} \int_s^{s+t} A(r)\dr \dss \dt \\	
		& \leq \left\|A\right\|_{L^2(0,T)} \liminf_{h \searrow 0 } \frac{1}{h} \int_0^h  \int_0^{T-t} \sqrt{t} \dss \dt =0.
		\end{align*}
	By the extension of the Aubin-Lions Lemma \ref{LmLmConv} we get the desired convergence result for every compact set $\Omega $ and for every finite time horizon $T$. 
	
	At last we prove that the limit of $\rho_{\omega_n}$ solves the time averaged Fokker-Planck equation in a weak sense. Therefore, we calculate the limit, as $n \rightarrow \infty$, in the weak formulations of each $\rho_{\omega_n}$
					\begin{align}
					\begin{split}
			 &  \iiint\limits_{\left[0,T\right]\times\Rd\times \Rd} \nabla_x W_{\omega_n t} (x,y) \nabla \psi(t,x) \rho_{\omega_n}  (t,x)  \rho_{\omega_n} (t,y)  \mathrm{d} x \mathrm{d} y\mathrm{d} t -\int_0^{T}\int_{\Rd}  \Delta \psi(t,x)  \rho_{\omega_n} (t,x)^m \dx \dt \\
			= & \int_0^T \int_\Rd \rho_{\omega_n}  (t,x) \partial_t \psi(t,x) \dx \dt + \int_\Rd \rho_0(x) \psi(0,x)\dx.
			\end{split} \label{HFWF}
			\end{align}
			To deduce the limit of the first integral in \eqref{HFWF} we fix some $R>0$ and split the domain of integration in to $\mathds{B}_R \times \mathds{B}_R $ and the complement. Note, the solutions $\rho_\omega$ are uniformly bounded in $L^\infty(0,T;L^m(\Rd))$. This yields that $\rho_{\omega_n} (t,x)\rho_{\omega_n} (t,y)$ converges to $\rho_\infty(t,x)\rho_\infty(t,y)$ in $L^m(0,T;L^m(\mathds{B}_R \times \mathds{B}_R))$ and therefore, we also have that $\rho_{\omega_n} (t,x)\rho_{\omega_n} (t,y)$ converges to $\rho_\infty(t,x)\rho_\infty(t,y)$ in $L^1(0,T;L^1(\mathds{B}_R \times \mathds{B}_R))$. Furthermore, $\nabla W_{{\omega_n}  t} (x,y) \rightharpoonup^* \nabla \overline{W} (x,y)$ in $L^{\infty}(0,T)$ for every $x,y\in\Rd$ (see, for instance, \cite{braides2002}). Thus, also $\nabla W_{{\omega_n}  t}  \rightharpoonup^*   \nabla \overline{W}$ in $L^{\infty}(0,T;L^\infty(\mathds{B}_R \times \mathds{B}_R))$. In combination we have  
\begin{equation*}
	\nabla W_{{\omega_n}  t}(x,y) \rho_{\omega_n} (t,x)\rho_{\omega_n} (t,y)  \rightharpoonup^*  \nabla \overline{W}(x,y) \rho_\infty(t,x)\rho_\infty(t,y) \qquad  \mathrm{in} \  L^{\infty}(0,T;L^\infty(\mathds{B}_R \times \mathds{B}_R))
	\end{equation*}
	and since the domain of integration is compact we finally have
			\begin{align*}
			&\lim_{\ntoinf} \iiint\limits_{\left[0,T\right]\times\mathds{B}_R \times \mathds{B}_R} \nabla_x W_{{\omega_n} t} (x,y) \nabla \psi(t,x) \rho_{\omega_n}  (t,x)  \rho_{\omega_n} (t,y) \mathrm{d} x \mathrm{d} y\mathrm{d} t \\\quad&= \iiint\limits_{\left[0,T\right]\times\mathds{B}_R \times \mathds{B}_R } \nabla_x \overline{W} (x,y) \nabla \psi(t,x) \rho_\infty (t,x)  \rho_\infty (t,y)  \mathrm{d} x \mathrm{d} y\mathrm{d} t .
			\end{align*}
			To calculate the corresponding integral over the complement of $\mathds{B}_R \times \mathds{B}_R$, choose $R$ sufficiently large, such that  $ supp \ \psi \subset \mathds{B}_R$. Again, apply assumption (W5) to obtain
\begin{align*}
& \limsup_\ntoinf \left| \ \ \iiint\limits_{\left[0,T\right]\times ( \mathds{B}_R \times \mathds{B}_R )^c  } \nabla_x W_{{\omega_n} t} (x,y) \nabla \psi(t,x) \rho_{\omega_n}  (t,x)  \rho_{\omega_n} (t,y) \mathrm{d} x \mathrm{d} y\mathrm{d} t \right| \\
 \leq & \limsup_\ntoinf \left\| \nabla \psi \right\|_\infty d_4  \int_0^T \left(\int_{\mathds{B}_R} \rho_{\omega_n}(t,x) \dx + R^{r-2} M(\rho_{\omega_n}(t))\right) \dt \\
 \leq & \left\| \nabla \psi \right\|_\infty d_4 \int_0^T \left(\int_{\mathds{B}_R^c} \rho_{\infty}(t,x) \dx + R^{r-2} C(T,\tau_*,\rho_0 ) \right) \dt,
\end{align*}
 where we used  in the last step Lebesgue's convergence theorem, since the integrand converges pointwise (due to the narrow convergence) and since the second moments are bounded uniform, hence also the integrand is bounded. Moreover, the integrand in the last inequality converges pointwise to zero, as $R \rightarrow \infty$, is bounded by a constant, and hence by again invoking Lebesgue's theorem we have
	\begin{equation*}
	\lim_\ntoinf  \iiint\limits_{\left[0,T\right]\times ( \mathds{B}_R \times \mathds{B}_R )^c   } \nabla_x W_{{\omega_n} t} (x,y) \nabla \psi(t,x) \rho_{\omega_n}  (t,x)  \rho_{\omega_n} (t,y) \mathrm{d} x \mathrm{d} y\mathrm{d} t  = 0.
	\end{equation*}
			The limit of the second integral in \eqref{HFWF} follows from the $L^m(0,T;L^m(\Omega))$-convergence of $\rho_{\omega_n}$ and we get
				\begin{equation*}
					\lim_{\ntoinf} \int_0^{T}\int_{\Rd}  \Delta \psi(t,x)  \rho_{\omega_n} (t,x)^m \dx \dt = \int_0^{T}\int_{\Rd}  \Delta \psi(t,x)  \rho_\infty (t,x)^m \dx \dt.
				\end{equation*}	
By the narrow convergence of $\rho_{\omega_n}$ one can show that the limit of the first integral of the left-hand-side of \eqref{HFWF} is equal to
	\begin{equation*}
		\lim_{\ntoinf}\int_0^T \int_\Rd \rho_{\omega_n}  (t,x) \partial_t \psi(t,x) \dx \dt = \int_0^T \int_\Rd \rho_\infty  (t,x) \partial_t \psi(t,x) \dx \dt.
	\end{equation*}
Summarized, we obtain that $\rho_\infty$ solves
	\begin{align*}
	 			 &  \iiint\limits_{\left[0,T\right]\times\Rd\times \Rd} \nabla_x \overline{W} (x,y) \nabla \psi(t,x) \rho_\infty  (t,x)  \rho_\infty   (t,y)  \mathrm{d} x \mathrm{d} y\mathrm{d} t -\int_0^{T}\int_{\Rd}  \Delta \psi(t,x)  \rho_\infty   (t,x)^m \dx \dt \\
			= & \int_0^T \int_\Rd \rho_\infty    (t,x) \partial_t \psi(t,x) \dx \dt + \int_\Rd \rho_0(x) \psi(0,x)\dx,
	\end{align*}
	yielding that $\rho_\infty$ is a weak solution of the averaged Fokker-Planck equation \eqref{limFPeq}.
\end{proof}


\appendix
\section{The Moreau-Yosida approximation}\label{app:moryos}

In this part we complete the technical gaps in the proof of the existence of curves of steepest descents concerning the Moreau-Yosida approximation and the resolvent. The first result is an auxiliary inequality which will be used several times to derive lower bounds for the Moreau-Yosida approximation or upper bounds for the squared distance.


\begin{lemma}\label{BoundsMY}
Let $c_*,\tau_*$ be the constants from Assumptions \ref{DefRegE} and \ref{DefRegP}. Then, for all $\tau \in \left( 0, \tau_* \right), \ t \in \left[0,\infty\right)$ and all $u,v\in \X$, we have that:
	\begin{align}
			\phi(\tau,t,u) &\geq c_* - \frac{1}{\tau_*- \tau} \bd^2(u_*,u) , \label{LBoundMY} \\
			\bd^2(v,u) & \leq \frac{4\tau \tau_*}{\tau_* - \tau} \left( \Phi(\tau,t,u;v) - c_*  + \frac{1}{\tau_* - \tau} \bd^2(u_*,u) \right). \label{UBoundMY}
	\end{align}
\end{lemma}

\begin{proof}
We use the Cauchy-type inequality 
	\begin{equation*}
		(a+b)^2 \leq \left(1+ \varepsilon \right) a^2 + \left(1+ \frac{1}{\varepsilon}\right) b^2 \qquad \forall \ a,b \geq 0 , \ \varepsilon >0,
	\end{equation*}
	with $a= d(v,u), \ b= d(u_*,u)$ and $\varepsilon = \frac{\tau_* - \tau}{\tau_* + \tau}$ , to get:
		\begin{equation*}
			\frac{1}{2\tau_*} \bd^2(u_*,v) \leq  \frac{1}{\tau_* + \tau} \bd^2(v,u) + \frac{1}{\tau_* - \tau} \bd^2(u_*,u).
		\end{equation*}
	This yields for every $u,v \in \X$ and $\tau < \tau_*$:
\begin{align*}
				\Phi(\tau,t,u; v) & = \frac{\tau_* - \tau }{2\tau(\tau_* + \tau)} \bd^2(u,v) + \frac{1}{\tau_* +\tau} \bd^2(u,v) + \mathcal{E}(v) + \mathcal{P}_t(v) \\
				& \geq \frac{\tau_* -\tau}{4\tau_* \tau} \bd^2(u,v) + \frac{1}{2\tau_*} \bd^2(u_*,v)  - \frac{1}{\tau_* - \tau} \bd^2(u_*,u) + \mathcal{E}(v) + \mathcal{P}_t(v)\\
				& \geq \frac{\tau_* -\tau}{4\tau_* \tau} \bd^2(u,v)   - \frac{1}{\tau_* - \tau} \bd^2(u_*,u) + c_*,
			\end{align*}
from which \eqref{UBoundMY} and, after taking the infimum with respect to $v\in \X$, \eqref{LBoundMY} follows.
\end{proof}


	The next lemma gives an a priori estimate for the distance of a minimizer at time $s$ and with regularization parameter $\sigma$ in terms of a minimizer at time $t$ and with regularization parameter $\tau$.
	
		\begin{lemma} \label{EstimateMinimizers}
		Let $u \in \X$ and define $u_{\sigma,s} \in J_{\sigma,s}(u)$ and $u_{\tau,t} \in J_{\tau,t}(u)$ with $\sigma < \tau$ and $s< t$, then it holds that
			\begin{equation*}
				\bd^2(u,u_{\sigma,s}) \leq \bd^2(u,u_{\tau,t}) + \frac{2\tau\sigma}{\tau-\sigma} \int_s^t \alpha(r) \dr \left( 2 +\bd^2( u_*,u_{\sigma,s}) + \bd^2( u_*,u_{\tau,t}) \right).
			\end{equation*}
		\end{lemma}
		
\begin{proof}
Fix $u_{\sigma,s} \in J_{\sigma,s}(u)$ and $u_{\tau,t} \in J_{\tau,t}(u)$ with $\sigma < \tau, s <t$ and exploit once again the variational definition of the resolvent and of the Moreau-Yosida approximation to get 
			\begin{align*} 
			\Phi( \sigma, s,u; u_{\sigma,s}) & \leq \Phi( \sigma, s,u; u_{\tau,t}) \\
			& = \left( \frac{1}{2\sigma} - \frac{1}{2\tau} \right) \bd^2(u,u_{\tau,t}) + \Phi(\tau,s,u;u_{\tau,t}) \\
			& = \left( \frac{1}{2\sigma} - \frac{1}{2\tau} \right) \bd^2(u,u_{\tau,t}) + \Phi(\tau,t,u;u_{\tau,t}) + \mathcal{P}_s(u_{\tau,t}) - \mathcal{P}_t(u_{\tau,t}) \\
			& \leq \left( \frac{1}{2\sigma} - \frac{1}{2\tau} \right) \bd^2(u,u_{\tau,t}) + \Phi(\tau,t,u;u_{\sigma,s}) + \mathcal{P}_s(u_{\tau,t}) - \mathcal{P}_t(u_{\tau,t}).
			\end{align*}
			Subtract $\Phi(\tau,t,u;u_{\sigma,s})$ from both sides to obtain
			\begin{equation*}
			\left( \frac{1}{2\sigma} - \frac{1}{2\tau} \right) \bd^2(u,u_{\sigma,s}) + \mathcal{P}_s(u_{\sigma,s}) - \mathcal{P}_t(u_{\sigma,s})	\leq \left( \frac{1}{2\sigma} - \frac{1}{2\tau} \right) \bd^2(u,u_{\tau,t}) + \mathcal{P}_s(u_{\tau,t}) - \mathcal{P}_t(u_{\tau,t}).
			\end{equation*}
			Since $\sigma < \tau $, we can multiply with $\frac{2\tau\sigma}{\tau-\sigma}$ to get
				\begin{align*}
					\bd^2(u,u_{\sigma,s})  & \leq \bd^2(u,u_{\tau,t}) + \frac{2\tau\sigma}{\tau-\sigma} \left( \mathcal{P}_s(u_{\tau,t}) - \mathcal{P}_t(u_{\tau,t}) + \mathcal{P}_t(u_{\sigma,s}) - \mathcal{P}_s(u_{\sigma,s}) \right)\\
					& \leq  \bd^2(u,u_{\tau,t}) + \frac{2\tau\sigma}{\tau-\sigma} \int_s^t \alpha(r) \dr \left( 2 +\bd^2( u_*,u_{\sigma,s}) + \bd^2( u_*,u_{\tau,t}) \right) ,
				\end{align*}
			where in the last step, we used the absolute continuity of $\mathcal{P}_t$. 
			\end{proof}

	The resolvent possesses a continuity property, i.e., if we send the regularization parameter $\tau$ to zero, then every minimizer of $\Phi(\tau,t,u,\cdot)$ converges to $u$. This continuity property even holds true, if we take the limit with respect to time $t$ and to space $u$.

	\begin{lemma}[Continuity property of the resolvent]\label{ContofResolvent}
	Let  $u \in \mathcal{D}\left( \mathcal{E} \right)$ and  $\tau \in \left(0,\tau_* \right)$.	Given the convergent sequences $\tau_n \searrow 0,  \ t_n \rightarrow t$ and $u_n \dar u$, define a sequence of minimizers $v_n \in J_{\tau_n,t_n}(u_n)$. If in addition $\mathcal{E} (u_n) \leq C $,  then we have
		\begin{equation*}
				v_n \dar u \qquad \mathrm{as} \quad  n \rightarrow \infty.
		\end{equation*}	
	\end{lemma}


\begin{proof}		
We can assume without loss of generality that $\tau_n < \tau_*$. Use the monotonicity of the Moreau-Yosida approximation \eqref{MYmonotonicity} and the estimate \eqref{UBoundMY} with $v=v_n$ and $u=u_n$, to obtain:
\begin{align*}
\bd^2(v_n,u_n) & \leq \frac{4\tau_n \tau_*}{\tau_* - \tau_n} \left( \Phi(\tau_n,t_n,u_n;v_n) - c_*  + \frac{1}{\tau_* - \tau_n} \bd^2(u_*,u_n) \right)\\
& \leq \frac{4\tau_n \tau_*}{\tau_* - \tau_n} \left( \mathcal{E}(u_n) + \mathcal{P}_{t_n}(u_n) - c_*  + \frac{1}{\tau_* - \tau_n} \bd^2(u_*,u_n) \right).
\end{align*}
By $\bs$-continuity $\mathcal{P}_{t_n}(u_n)$ is bounded, and by assumption, also $\mathcal{E}(u_n)$ is bounded from above, so we can further estimate to obtain
	\begin{align*}
	 \bd^2(v_n,u_n) & \leq  \frac{4\tau_n \tau_*}{\tau_* - \tau_n} \left( C - c_*  + \frac{1}{\tau_* - \tau_n} \bd^2(u_*,u_n) \right)  .
	\end{align*}
	 Taking the limit $\ntoinf$ yields the desired convergence $v_n \dar u$.
	\end{proof}

The next lemma shows that the Moreau-Yosida approximation is also continuous, i.e., the minimal value of the Moreau-Yosida functional depends continuously on the regularization parameter $\tau$, the time coordinate $t$ and the spatial coordinate $u$.

\begin{lemma}[Continuity property of the Moreau-Yosida approximation] \label{MoreauCont}
The map $	\left(\tau,t,u\right) \mapsto \phi(\tau,t,u)$	is $\bd$-continuous on $\left(0, \tau_* \right) \times \left[0,\infty \right) \times \mathcal{D} \left( \mathcal{E} \right)$. 
	\end{lemma}
	
	
	\begin{proof}
	Choose a sequence $\left(\tau_n,t_n,u_n\right)_\ninN$ in $ \left(0,\tau_* \right) \times \left[0,\infty \right) \times \mathcal{D} \left( \mathcal{E} \right) $ with $\tau_n \rightarrow \tau \in \left(0,\tau_* \right), t_n \rightarrow t \in \left[0,\infty \right)$ and $u_n \dar u \in \mathcal{D} \left( \mathcal{E} \right) $. Then, it follows for an arbitrary $v\in \X$ that
	\begin{align*}
	\limsup_{\ntoinf} \phi(\tau_n,t_n,u_n) & \leq \limsup_{\ntoinf} \Phi(\tau_n,t_n,u_n;v)  = \limsup_{\ntoinf} \frac{1}{2\tau_n} \bd^2(u_n,v) + \mathcal{E}(v) + \mathcal{P}_{t_n}(v) = \frac{1}{2\tau} \bd^2(u,v) + \mathcal{E}(v) + \mathcal{P}_{t}(v). 
	\end{align*}
	Taking the infimum over $v$ on the r.h.s. yields the upper semi-$\bd$-continuity of $\phi$. 
	
	To prove the lower semi-continuity, choose $v_n \in J_{\tau_n,t_n}(u_n)$ and first of all notice that this sequence is bounded, since by the the upper estimate for the Moreau-Yosida approximation \eqref{UBoundMY}
	\begin{align*}
		\bd^2(v_n,u_n) & \leq \frac{4\tau_n \tau_*}{\tau_* - \tau_n} \left( \Phi(\tau_n,t_n,u_n;v_n) - c_*  + \frac{1}{\tau_* - \tau_n} \bd^2(u_*,u_n) \right) \\
		& =  \frac{4\tau_n \tau_*}{\tau_* - \tau_n} \left( \phi(\tau_n,t_n,u_n) - c_*  + \frac{1}{\tau_* - \tau_n} \bd^2(u_*,u_n) \right).
	\end{align*}	
	 Since every term on the r.h.s. is bounded the sequence $\left(v_n \right)_\ninN$ is $\bd$-bounded and by the continuity of the resolvent (Lemma \ref{ContofResolvent}) we also have that $v_n$ converges to $u$ in $\bd$. Now, the continuity of $\mathcal{P}_t$ and the variational definition of $v_n$ yield the lower semi-$\bd$-continuity:
	\begin{align*}
	\liminf_{\ntoinf} \phi(\tau_n,t_n,u_n) & = \liminf_{\ntoinf} \left(\frac{1}{2\tau_n} \bd^2(u_n,v_n) + \mathcal{E}(v_n) + \mathcal{P}_{t_n}(v_n)	\right) \\
	& \geq \liminf_{\ntoinf} \left(\frac{1}{2\tau_n} \left( \bd^2(u_n,u) + \bd^2(u,v_n) - 2 d(u_n,u)d(u,v_n) \right) + \mathcal{E}(v_n) + \mathcal{P}_{t_n}(v_n)\right)	 \\
	& = \liminf_{\ntoinf} \left(\frac{1}{2\tau}  \bd^2(u,v_n) + \mathcal{E}(v_n) +\mathcal{P}_{t}(v_n)\right) \\
	& \geq \liminf_{\ntoinf} \phi(\tau,t,u) = \phi(\tau,t,u). \qedhere
	\end{align*}
	\end{proof}

	 The next lemma shows that we can extend the Moreau-Yosida approximation continuously at $\tau=0$ and that the value at the boundary is equal to the starting functional, thus the Moreau-Yosida approximation is an approximation for the original functional.
	
	\begin{lemma}[Approximation property of the Moreau-Yosida approximation] \label{MYApproxProp}
	For all $u \in  \mathcal{D} \left( \mathcal{E} \right)$ and all sequences $\left( \tau_n, t_n \right) \subset \left(0, \tau_* \right) \times \left[ 0 ,\infty \right)$ with $\tau_n \searrow 0$ and $t_n \rightarrow t$, we have that
			\begin{equation*}
			\lim_{\ntoinf } \phi( \tau_n, t_n , u) = \mathcal{E}(u) + \mathcal{P}_t(u). 
		\end{equation*}
	\end{lemma}

	\begin{proof}	
Given $u \in \mathcal{D}\left(\mathcal{E}\right)$, choose an element $u_n \in J_{\tau_n,t_n}(u)$. Thus, by the continuity of the resolvent (Lemma \ref{ContofResolvent}), $u_n \dar u$ as $\ntoinf$. The lower semi-$\bs$-continuity of $\mathcal{E}$ and the $\bs$-continuity of $\mathcal{P}_t$ yield a lower bound for the limit, i.e.:
	\begin{align*}
	\liminf_{\ntoinf} \phi(\tau_n, t_n, u) & =  \liminf_{ \ntoinf} \frac{1}{2 \tau_n} \bd^2(u_n,u) + \mathcal{E}(u_n) +\mathcal{P}_{t_n}(u_n)  \geq \liminf_{ \ntoinf } \mathcal{E}(u_n) +\mathcal{P}_{t_n}(u_n)  \geq  \mathcal{E}(u) +\mathcal{P}_t(u).
	\end{align*}
	The reverse inequality follows from the monotonicity of the Moreau-Yosida approximation \eqref{MYmonotonicity}, and with the $\bs$-continuity of $\mathcal{P}_t$:
	\begin{equation*}
	\limsup_{\ntoinf} \phi(\tau_n, t_n, u) \leq \limsup_{\ntoinf}  \mathcal{E}(u) +\mathcal{P}_{t_n}(u)  = \mathcal{E}(u) +\mathcal{P}_{t}(u). \qedhere
	\end{equation*}
	\end{proof}

	Next, we prove a certain differentiability property of the \emph{Moreau-Yosida approximation}, which will be important later to derive a discrete energy inequality for the discrete solutions of the Minimizing Movement scheme.  


\begin{lemma}[Joint differentiability of the Moreau-Yosida approximation] \label{DiffofPhi}
For every $u \in \mathcal{D}\left(\mathcal{E} \right)$ and $ t \in \left[0,\infty \right) $, the map $\tau \rightarrow \phi(\tau,t + \tau,u)$ is locally Lipschitz continuous on $\left(0,\tau_*\right)$ and differentiable except on a countable set $\mathcal{S}_{t,u}$. For every $\tau \in \left( 0 , \tau_* \right) \backslash \mathcal{S}_{t,u} $ we have:
		\begin{align}
			\frac{d}{d\tau} \phi(\tau,t+\tau,u)  =  -\frac{1}{2\tau^2} \bd^2(u,v) + \partial_t \mathcal{P}_{t+\tau}(v) \qquad \forall \ v\in J_{\tau,t+\tau}(u).  \label{DifferentialofPhi3}
		\end{align}
\end{lemma}


\begin{proof}
Fix $\sigma < \tau$ in $ \left(0, \tau_* \right)$ and choose $u_\tau \in J_{\tau,t+\tau}(u)$ and $u_\sigma \in J_{\sigma,t+\sigma}(u)$ and exploit the variational definition of the Moreau-Yosida approximation to obtain
		\begin{align*}
\phi(\tau,t+\tau,u)-\phi(\sigma,t+\sigma,u) & \leq \Phi(\tau,t+\tau, u, u_\sigma) - \Phi(\sigma,t+\sigma, u, u_\sigma) \\
&= \left( \frac{1}{2\tau} - \frac{1}{2\sigma} \right) \bd^2(u,u_\sigma ) + \mathcal{P}_{t+\tau} (u_\sigma) - \mathcal{P}_{t+\sigma} (u_\sigma) \\
& \leq \frac{\sigma - \tau}{2 \tau \sigma}  \bd^2(u,u_\sigma ) +   \int_{t+\sigma}^{t+\tau} \alpha(r) \dr ( 1+ \bd^2(u_* , u_\sigma) ).
\end{align*}
Analogously, a lower bound can be established by reversing the role of $u_\tau$ and $u_\sigma$:
\begin{align*}
\phi(\tau,t+\tau,u)-\phi(\sigma,t+\sigma,u) & \geq \Phi(\tau,t+\tau, u, u_\tau) - \Phi(\sigma,t+\sigma, u, u_\tau) \\
&= \left( \frac{1}{2\tau} - \frac{1}{2\sigma} \right) \bd^2(u,u_\tau ) + \mathcal{P}_{t+\tau} (u_\tau) - \mathcal{P}_{t+\sigma} (u_\tau) \\
& \geq \frac{\sigma - \tau}{2 \tau \sigma} \bd^2(u,u_\tau )  -   \int_{t+\sigma}^{t+\tau} \alpha(r) \dr ( 1+ \bd^2(u_* , u_\tau) ).
	\end{align*}
Note that by estimate \eqref{UBoundMY} and by the monotonicity of the Moreau-Yosida approximation we have 
\begin{align*}
\bd^2(u_{\tau}, u) & \leq \frac{4\tau \tau_*}{\tau_* - \tau} \left( \Phi(\tau,t+\tau,u,u_{\tau}) - c_*  + \frac{1}{\tau_* - \tau} \bd^2(u_*,u) \right) \\
& = \frac{4\tau \tau_*}{\tau_* - \tau} \left( \phi(\tau,t+\tau,u ) - c_*  + \frac{1}{\tau_* - \tau} \bd^2(u_*,u) \right) \\
& \leq  \frac{4\tau \tau_*}{\tau_* - \tau} \left( \mathcal{E}(u) + \mathcal{P}_{t+ \tau}(u)- c_*  + \frac{1}{\tau_* - \tau} \bd^2(u_*,u) \right) \\
& \leq  \frac{4\tau \tau_*}{\tau_* - \tau} \left( \mathcal{E}(u) + \sup_{\tau \in \left(0,\tau_* \right) } \mathcal{P}_{t+ \tau}(u)- c_*+ \frac{1}{\tau_* - \tau} \bd^2(u_*,u) \right).
\end{align*}
Thus, $\bd^2(u_{\tau}, u)$, $\bd^2(u_{\tau}, u_*)$, $\bd^2(u_{\sigma}, u)$, and $\bd^2(u_{\sigma }, u_*)$ are locally bounded by some constant independent of $\tau$ and $\sigma$ and therefore $\tau \mapsto \phi(\tau,t+\tau,u)$ is locally absolutely continuous. To calculate the derivative for a point $\tau \in \left(0,\tau_* \right) \backslash \mathcal{S}_{t,u}$ in the set of differentiability, divide the previous inequalities by $\tau - \sigma $ such that we have for $ \sigma < \tau$:
\begin{align*}
-\frac{1}{2\tau\sigma} \bd^2(u,u_\tau) + \frac{\mathcal{P}_{t+\tau} (u_\tau) - \mathcal{P}_{t+\sigma} (u_\tau)}{\tau-\sigma} & \leq \frac{\phi(\tau,t+\tau,u)-\phi(\sigma,t+\sigma,u)}{\tau-\sigma}.
\end{align*}
The left-sided limit $\sigma \nearrow \tau$ yields a lower bound for the derivative $\tau \mapsto \phi(\tau,t+\tau,u)$ and analogously we gain from the inequality for $\tau < \sigma$ and the right-sided limit $\sigma \searrow \tau$ an upper bound. Since $u_\tau$ was arbitrarily chosen in the resolvent, the value of the derivative is independent of the $u_\tau$ and therefore the desired formula is true. 
\end{proof}

		
		To conclude this section about the \emph{Moreau-Yosida approximation} and the \emph{resolvent} we state a known result about an a priori result for the local slope in terms the minimizer of the \emph{Moreau-Yosida approximation}, see for instance \cite{ags}.
		
		\begin{lemma}\label{EstimateLocalSlope}
			Given $u\in \X, \tau >0, t \in \left[0,\infty \right)$ and $v\in J_{\tau,t}(u)$, we have
				\begin{equation}
					| \partial (\mathcal{E} + \mathcal{P}_t) | (v) \leq \frac{1}{\tau} \bd(u,v).	\label{LocalSlopeEstimate}	
				\end{equation}
		\end{lemma}
	
\section{The Minimizing Movement scheme}\label{app:minmov}

In this section, we derive the classical estimates of Theorem \ref{BoundsMMS} and the discrete energy inequality of Theorem \ref{DiscIntInequMMs}.


\noindent\emph{Proof of Theorem \ref{BoundsMMS}:}
	For a given partition $\mathcal{T}$ with $\btau \in \left(0, \tau_*\right) $ consider the discrete solution $\left(\uk\right)_\kinN$ obtained by the Minimizing Movement scheme. Since $\uk$ is a minimizer for $\Phi(\tau,t^k_\btau,\ukm; \cdot)$ it satisfies the discrete variational inequality:
			\begin{equation*}
				\frac{1}{2\tau_k} \dk + \ek + \pk \leq \frac{1}{2\tau_k} \bd^2(\ukm,\ukm) +\ekm + \pkm.
			\end{equation*}
		Rearrange and sum these inequalities from $k=1$ to $k=N$ and exploit the telescopic sum to obtain 
		\begin{align}
		\sum_{k=1}^N \frac{1}{2\tau_k } \dk & \leq \mathcal{E}(u_0) - \ekn + \sum_{k=1}^N \left( \pkm - \pk\right) \label{AuxClasEstMMs} \\
		& =  \mathcal{E}(u_0) - \ekn + \mathcal{P}_{0} (u_0) - \mathcal{P}_{t^N_\btau}( \ukn) +  \sum_{k=0}^{N-1}\left(  \mathcal{P}_{t^{k+1}_\btau} ( \uk) - \mathcal{P}_{t^k_\btau} (\uk)\right) \notag  \\
		&\leq \mathcal{E}(u_0) + \mathcal{P}_{0} (u_0) + \frac{1}{2\tau_*}\bd^2(u_*, \ukn)  - c_*  +   \sum_{k=0}^{N-1} \int_{t^k_\btau}^{t^{k+1}_\btau} \alpha (r) \dr \left( 1 + \bd^2(u_*,\uk) \right) \label{AuxIneqBoundsMMs}.
		\end{align}
Furthermore, using Young's inequality with $\varepsilon = \frac{\tau_*}{2}$, we get
					\begin{align*}
				\frac{1}{2}	\dksn - \frac{1}{2} \dksz & = \sum_{k=1}^N \frac{1}{2}\dks - \frac{1}{2} \dksm  \\
				& \leq \sum_{k=1}^N \bd(\uk,\ukm) \bd(\uk,u_*)   \\
				& \leq \frac{\tau_*}{2}  \sum_{k=1}^N \frac{1}{2\tau_k } \dk +  \frac{1}{\tau_*}\sum_{k=1}^N \tau_k \dks.
				\end{align*}
	Inserting the auxiliary inequality \eqref{AuxIneqBoundsMMs} from above yields
				\begin{align*}
				\frac{1}{2}	\dksn - \frac{1}{2} \dksz & \leq \frac{\tau_*}{2} \left( \mathcal{E}(u_0)  + \mathcal{P}_{0} (u_0) + \frac{1}{2\tau_*} \bd^2(u_*, \ukn)- c_* +   \sum_{k=0}^{N-1} \int_{t^k_\btau}^{t^{k+1}_\btau} \alpha (r) \dr \left( 1 + \bd^2(u_*,\uk) \right)  \right) \\
				& \quad + \frac{1}{\tau_*} \sum_{k=1}^N \tau_k \dks \\
				& \leq \frac{\tau_*}{2}  \left( \mathcal{E}(u_0)  + \mathcal{P}_{0} (u_0) - c_* + \int_0^T \alpha(r) \dr (1 + \dksz) \right) \\
				& \quad + \frac{1}{4} \dksn + \sum_{k=1}^N  \left( \frac{\tau_*}{2}  \int_{t^k_\btau}^{t^{k+1}_\btau} \alpha (r) \dr   +\frac{\tau_k}{\tau_*} \right) \dks.
			\end{align*}
		Rearrange the inequality to obtain
		\begin{align*}
		\dksn & \leq 2 \dksz + 2 \tau_*  \left( \mathcal{E}(u_0)  + \mathcal{P}_{0} (u_0) - c^* + \int_0^T \alpha(r) \dr (1 + \dksz) \right) \\
		& \quad + 4 \sum_{k=1}^N  \left( \frac{\tau_*}{2} \int_{t^k_\btau}^{t^{k+1}_\btau} \alpha (r) \dr   +\frac{ \tau_k}{\tau_*} \right) \dks \\
		& =: \tilde{C}(T,\tau_*, u_0) + 4 \sum_{k=1}^N \alpha_k \dks.
		\end{align*}
		Since by assumption $\sup_k 4 \alpha_k < 1$ one can apply the discrete version of Gronwall's lemma \cite[Lemma 3.2.4]{ags} to conclude 
		\begin{equation*}
				\dksn \leq \hat{C} (T,\tau_*, u_0)  \exp \left[ \hat{c} (T,\tau_*, u_0) \sum_{k=1}^{N-1} \alpha_k \right] \leq \hat{C} (T,\tau_*, u_0)  \exp \left[ \hat{c} (T,\tau_*, u_0) \left( \frac{\tau_*}{2} \int_0^T \alpha(r) \dr + \frac{T}{\tau_*}  \right)  \right].
		\end{equation*}
Hence, we have proven the $\bd$-boundedness of the discrete solution \eqref{bounddn}. With this result and with the first chain of inequalities we can deduce the first estimate \eqref{boundsumdk}, i.e.,
		\begin{align*}
		\sum_{k=1}^N \frac{1}{2\tau_k } \dk & \leq \mathcal{E}(u_0) + \mathcal{P}_{0} (u_0) + \frac{1}{2\tau_*}\bd^2(u_*, \ukn)  - c_*  +   \sum_{k=0}^{N-1} \int_{t^k_\btau}^{t^{k+1}_\btau} \alpha (r) \dr \left( 1 + \bd^2(u_*,\uk) \right) \\
		& \leq \mathcal{E}(u_0) + \mathcal{P}_{0} (u_0) + \frac{1}{2\tau_*}C(T,\tau_*, u_0)  - c_*  + \int_0^T \alpha(r) \dr \left( 1 + C(T,\tau_*, u_0)  \right).
		\end{align*}
Again, using this inequality yields the upper bound for $\ekn$, since
\begin{align*}
	\ekn & \leq \ekn + \sum_{k=1}^N \frac{1}{2\tau_k } \dk \\
	& \leq  \mathcal{E}(u_0) + \mathcal{P}_{0} (u_0) - \mathcal{P}_{t^N_\btau}( \ukn) +  \sum_{k=0}^{N-1}  \mathcal{P}_{t^{k+1}_\btau} ( \uk) - \mathcal{P}_{t^k_\btau} (\uk) \\
	& \leq  \mathcal{E}(u_0) + \mathcal{P}_{0} (u_0)  + \frac{1}{2\tau_*}\bd^2(u_*, \ukn)  - c_* +   \sum_{k=0}^{N-1} \int_{t^k_\btau}^{t^{k+1}_\btau} \alpha (r) \dr \left( 1 + \bd^2(u_*,\uk) \right) \\
	& \leq  \mathcal{E}(u_0) + \mathcal{P}_{0} (u_0) +  \frac{1}{2\tau_*} C(T,\tau_*, u_0) - c_* + \int_0^T \alpha(r) \dr \left( 1 + C(T,\tau_*,u_0) \right). \tag*{\qed}
\end{align*}

The derivation of the discrete energy inequalities is combination of the theory of the Moreau-Yosida approximation and of the definition of the De Giorgi interpolation.


\noindent\emph{Proof of Theorem \ref{DiscIntInequMMs}:}
From Lemma \ref{DiffofPhi} we know that the map $\sigma \mapsto \phi(\sigma, t +\sigma,u)$ is locally absolutely continuous and we can compute the derivative at almost all $\sigma$, which is given by
		\begin{equation*}
			\frac{d}{d\sigma} \phi(\sigma,t+\sigma,u)= -\frac{1}{2\sigma^2} \bd^2(u,v) + \partial_t \mathcal{P}_{t+\sigma}(v) \qquad \forall \ v\in J_{\sigma,t+\sigma}(u). 
		\end{equation*}
	Choose $t=t^k_\btau$, $u=\uk$ and use $ \tilde{u}_{\btau} (t^k_\btau +\sigma) \in J_{\sigma, t^k_\btau +\sigma }(\uk) $ when integrating the equation with respect to $\sigma$ from $\varepsilon >0$ to $\tau_k$:
	\begin{equation*}
	\phi(\tau, t^{k+1}_\btau, \uk) - \phi( \varepsilon, t^k_\btau + \varepsilon, \uk)= \int_\varepsilon^{\tau_k} \left[- \frac{1}{2\sigma^2} \bd^2(\uk,\tilde{u}_{\btau} (t^k_\btau +\sigma) ) + \partial_t \mathcal{P}_{t^k_\btau + \sigma}(\tilde{u}_{\btau} (t^k_\btau +\sigma) )\right] \ds . 
	\end{equation*}
	Use $\tilde{u}_{\btau} (t^{k+1}_\btau ) \in J_{\tau, t^{k+1}_\btau }(\uk)$ and Lemma \ref{MYApproxProp} to perform the limit $\varepsilon \searrow 0$ to obtain
	\begin{align*}
	& \frac{1}{2\tau_k} \bd^2(\ukp, \uk) + \mathcal{E}(\ukp) + \mathcal{P}_{t^{k+1}_\btau} (\ukp) - \mathcal{E}(\uk) - \mathcal{P}_{t^k_\btau} (\uk) \\
	= &  \int_0^{\tau_k}\left[ - \frac{1}{2\sigma^2} \bd^2(\uk,\tilde{u}_{\btau} (t^k_\btau +\sigma) ) + \partial_t \mathcal{P}_{t^k_\btau +\sigma}(\tilde{u}_{\btau} (t^k_\btau +\sigma) )\right] \ds . 
	\end{align*}
	Apply Lemma \ref{EstimateLocalSlope} with $t=t^k_\btau+\sigma, u=\uk$ to obtain
		\begin{align*}
  & \frac{1}{2\tau_k} \bd^2(\ukp, \uk) + \mathcal{E}(\ukp) + \mathcal{P}_{t^{k+1}_\btau} (\ukp) - \mathcal{E}(\uk) - \mathcal{P}_{t^k_\btau} (\uk) \\
		 \leq & - \frac{1}{2} \int_0^{\tau_k} \left[ |  \partial\left( \mathcal{E} + \mathcal{P}_{t^k_\btau+\sigma} \right) |^2 ( \tilde{u}_\btau (t^k_\btau+\sigma) ) +  \partial_t \mathcal{P}_{t^k_\btau+\sigma}(\tilde{u}_{\btau} (t^k_\btau +\sigma) )\right] \ds.
	\end{align*}
	Summation from $k=0$ to $N -1 $ yields the desired discrete energy inequality.
\qed



\bibliographystyle{abbrv}
   \bibliography{thlit}

\end{document}